\documentclass[12pt,a4paper]{amsart}
\usepackage{amsfonts}
\usepackage{amsthm}
\usepackage{amsmath}
\usepackage{xcolor} 
\usepackage{amssymb}
\usepackage{mathtools} \usepackage{amscd}
\usepackage{dsfont} \usepackage[latin2]{inputenc}
\usepackage{t1enc}
\usepackage[mathscr]{eucal}
\usepackage{indentfirst}
\usepackage{graphicx}
\usepackage{graphics}
\PassOptionsToPackage{hyphens}{url}
\usepackage{hyperref}
\usepackage{epic}
\numberwithin{equation}{section}
\usepackage[margin=2.2cm, marginparwidth=2.4cm]{geometry}
\usepackage{epstopdf} 
\usepackage{enumitem}
\usepackage[sort,numbers]{natbib}

\usepackage{comment}
\numberwithin{equation}{section}
\newtheorem{assum}{Assumption}[section]
\newtheorem{theorem}{Theorem}[section]

\newtheorem{proposition}{Proposition}[section]
\newtheorem{remark}{Remark}[section]
\newtheorem{lemma}{Lemma}[section]

\DeclareMathOperator*{\argmin}{arg\,min}
\DeclareMathOperator*{\argmax}{arg\,max}

\def\mF{{\mathcal F}}
\def\mP{{\mathcal P}}
\def\mK{{\mathcal K}}

\def\mB{{\mathcal B}}
\def\mE{{\mathcal E}}
\def\mX{{\mathcal X}}
\def\mU{{\mathcal U}}
\def\mV{{\mathcal V}}
\def\mY{{\mathcal Y}}

\def\mL{{\mathcal L}}
\def\mI{{\mathcal I}}

\def\mH{{\mathcal H}}

\def\diam{\mathrm{Diam}}

\def\R{{\mathbb R}}

\def\mZ{{\mathcal Z}}

\def\gt{{\rightarrow}}

\def\NI{{\text{NI}}}

\def\eps{{\varepsilon}}

 \def\1{{\mathbf 1}}

\usepackage{todonotes}

\allowdisplaybreaks
\author{Yulong Lu}
\address{(YL)  Department of Mathematics and Statistics, Lederle Graduate Research Tower, University of Massachusetts, 710 N. Pleasant Street, Amherst, MA 01003.}
\email{yulonglu@umass.edu}

\begin{document}
\title[Two-Scale Mean-Field GDA Finds MNE]{Two-Scale Gradient Descent Ascent Dynamics Finds Mixed Nash Equilibria of Continuous Games: A Mean-Field Perspective}

\begin{abstract}
Finding the mixed Nash equilibria (MNE) of a two-player zero sum continuous game is an important and challenging problem in machine learning. A canonical algorithm to finding the MNE is the noisy gradient descent ascent method which in the infinite particle limit gives rise to the {\em Mean-Field Gradient Descent Ascent} (GDA) dynamics on the space of probability measures.  In this paper, we first study the  convergence  of a two-scale Mean-Field GDA dynamics for finding the MNE of the entropy-regularized objective. More precisely we show that for each finite temperature (or regularization parameter), the two-scale Mean-Field GDA with a suitable {\em finite} scale ratio  converges exponentially to the unique MNE without assuming the convexity or concavity of the interaction potential. The key ingredient of our proof lies in the construction of  new Lyapunov functions that dissipate exponentially along the Mean-Field GDA. We further study the simulated annealing of the Mean-Field GDA dynamics. We show that with a temperature schedule that decays logarithmically in time the annealed Mean-Field GDA converges to the MNE of the original unregularized objective.

\end{abstract}

\maketitle

\section{Introduction}

Minmax learning underpins numerous machine learning methods including Generative Adversarial Networks (GANs) \cite{goodfellow2020generative}, adversarial training \cite{madry2018towards} and reinforcement learning \cite{busoniu2008comprehensive}. The minmax learning is often be formulated as a zero-sum game between min and max players and hence can be formalized as a minmax optimization problem of the form
$$
\min_{x\in\mX} \max_{y\in \mY} K(x,y),
$$
where the function $K$ is the game objective, $x$ and $y$ represent the player strategies. When $K$ is nonconvex in $x$ or nonconcave in $y$, finding the pure Nash \cite{Nash1951} equilibria of $K$ is difficult and sometimes impossible since the pure Nash equilibrium may not even exist. On the other hand, the mixed Nash equilibria (MNE) \cite{glicksberg1952further} where the pure strategies are replaced by mixed strategies modeled by a probability distribution over the set of strategies, do exist for more general objective functions. Formally the mixed Nash equilibria consist of pairs of probability distributions $\mu$ and $\nu$ that solve 
\begin{equation}\label{eq:mne}
   \min_{\mu\in \mP(\mX)} \max_{\nu\in \mP(\mY)}  E_0(\mu,\nu), \text{ where }  E_0(\mu,\nu):= \int_{\mX} \int_{\mY} K(x,y) \mu(dx)\nu(dy).
\end{equation}
Here $\mP(\mX)$ and $\mP(\mY)$  denote the space of probability measures (or the set of all mixed strategies) on the state spaces $\mX$ and $\mY$ respectively. Thanks to Glicksberg's theorem \cite{glicksberg1952further}, a MNE exists if $\mX$ and $\mY$ are finite or if $\mX$ and $\mY$ are compact and $K$ is continuous. 
While MNE exist in much generality, it is still difficult to find them. Several progress have been made recently for finding MNE of high dimensional game problems with applications in GANs. For instance, the work \cite{hsieh2019finding} proposed a mirror-descent algorithm for finding MNE of \eqref{eq:mne} and applied the algorithms for Wasserstein GANs with provable convergence guarantees. The recent work \cite{domingo2020mean,ma2021provably} proposed and analyzed an entropy-regularized version of \eqref{eq:mne}:
\begin{equation}
\label{eq:minmax}
\min_{\mu\in \mP(\mX)} \max_{\nu\in \mP(\mathcal{Y})} E_{\tau}(\mu, \nu), \text{ where } E_{\tau}(\mu,\nu) := \int_{\mX} \int_{\mathcal{Y}} K(x,y) d\nu(y)d\mu(x) - \tau \mH(\mu) + \tau \mH(\nu).
\end{equation}
Here $\mH(\mu) = -\int \log \frac{d\mu}{dx} d\mu$ is the  entropy functional of the probability measure $\mu$, and the parameter $\tau>0$ is the entropy regularization parameter(or temperature). 
Observe that the objective energy functional $E_\tau$ is strongly convex in $\mu$ and strongly concave in $\nu$. As a result of the famous Von Neumann's minmax theorem \cite{v1928theorie,sion1958general,nikaido1955note,nikaido1954neumann}, one has that 
 \begin{equation}\label{eq:minmaxthm}
     \min_{\mu\in \mP(\mX)} \max_{\nu\in \mP(\mathcal{Y})} E_\tau (\mu, \nu) = 
     \max_{\nu\in \mP(\mathcal{Y})}  \min_{\mu\in \mP(\mX)} E_{\tau}(\mu, \nu).
 \end{equation}
 Under very mild assumptions on $K$, Problem  \eqref{eq:minmax} has a  unique MNE $(\mu^\ast, \nu^\ast)$ (see a proof in \cite{domingo2020mean}) in the sense that for any $\mu\in \mP(\mX), \nu \in \mP(\mY)$,
$$
E_\tau(\mu^\ast, \nu) \leq E_\tau(\mu^\ast,\nu^\ast) \leq E_\tau(\mu,\nu^\ast). 
$$
Furthermore, the MNE $(\mu^\ast, \nu^\ast)$ is given by the unique solution of the fixed point equations 
\begin{equation}\label{eq:fpt}
    \mu^\ast(dx)  \propto \exp\Big( - \int_{\mY} \tau^{-1} K(x,y)d\nu^\ast (y) \Big), \quad \nu^\ast(dy)   \propto \exp\Big( \int_{\mX} \tau^{-1} K(x,y)d\mu^\ast (x))\Big).
\end{equation}
In the setting where both players play finite mixtures of $N$ strategies, i.e. $\mu = \sum_{j=1}^N \frac{1}{N} \delta_{X^j}$ and   $\nu = \sum_{j=1}^N \frac{1}{N} \delta_{Y^j}$, the following noisy gradient descent-ascent dynamics is perhaps one of the most natural algorithms for finding MNE of \eqref{eq:minmax}:
\begin{equation}\label{eq:gdaparticle}
    dX^i_t = -\frac{1}{N} \sum_{j=1}^N \nabla_x K(X^i_t, Y^j_t) dt + \sqrt{2\tau} dW^i_t, \quad
dY^i_t = \frac{\eta}{N} \sum_{j=1}^N \nabla_y K(X^j_t, Y^i_t) dt + \sqrt{2\tau\eta} dB^i_t,
\end{equation}
where $W^i_t, B^i_t, i=1,\cdots N$ are independent Brownian motions. The parameter $\eta>0$ in \eqref{eq:gdaparticle} represents the ratio of timescales at which the gradient descent and ascent dynamics is undergoing. When $\eta = 1$, there is no timescale separation. However, GDA with different time scales is commonly used in minmax optimization \cite{jin2020local,lin2020gradient} and sometimes leads to better convergence property \cite{heusel2017gans}.  

In the limit of large number of strategies ($N \gt \infty$), the empirical measures $\mu_t^N = \frac{1}{N}\sum_{i=1}^N \delta_{X^i_t}$ and $\nu_t^N = \frac{1}{N} \sum_{i=1}^N \delta_{Y^i_t}$ of the interacting particle, with $\{X_0^i\}$ and $\{Y_0^i\}$ identically independent sampled from the initial distributions $\mu_0$ and $\nu_0$,  converge weakly within finite time to the solutions $\mu_t$ and $\nu_t$ of the {\em Mean-Field GDA} (also named the Interacting Wasserstein Gradient Flow in \cite{domingo2020mean}) dynamics:
\begin{equation}
\label{eq:WGDA}
\begin{aligned}
\partial_t \mu_t & = \nabla \cdot (\mu_t \int_{\mY} \nabla_x K(x,y)d\nu_t(y)) +\tau \Delta \mu_t,\\
\partial_t \nu_t & = \eta\Big(-\nabla \cdot (\nu_t \int_{\mX} \nabla_y K(x,y)d\mu_t(x)) + \tau\Delta \nu_t\Big).
\end{aligned}
\end{equation}
The dynamics \eqref{eq:WGDA} can be viewed as the minmax analogue of the Mean-Field Langevin \cite{nitanda2022convex,chizat2022mean,hu2021mean}  dynamics that arises naturally from the mean field analysis of optimization of two-layer neural networks. 

Despite its simplicity and popularity, the long-time convergence of the Mean-Field GDA  \eqref{eq:WGDA} is still not well understood, except in the extreme quasi-static \cite{ma2021provably} regime where the ascent dynamics is infinitely faster or slower than the
descent dynamics ($\eta =0 \text{ or } + \infty$). This motivates us to study the first  question, which to the best of our knowledge, remains open:

{\em \textbf{Question 1}: Does the  Mean-Field GDA  \eqref{eq:WGDA} with a finite scale ratio $\eta>0$ converge to the unique MNE? If so, what is the rate of convergence? }

We provide an affirmative answer to \textbf{Question 1} and establish a quantitative exponential convergence of \eqref{eq:WGDA} to the MNE for any fixed $\tau>0$. Furthermore, motivated by the recent simulated annealing results for Langevin dynamics \cite{tang2021simulated,raginsky2017non} in the context of global optimization or for Mean-Field Langevin dynamics \cite{chizat2022mean} in the setting of training two-layer neural networks, it is natural to ask what would happen to the dynamics \eqref{eq:WGDA} in the annealed regime that $\tau \gt 0$ as $t\gt\infty$. In particular, it is natural to ask 

{\em \textbf{Question 2:} Does the annealed Mean-Field GDA \eqref{eq:WGDA} with a decreasing temperature $\tau = \tau_t$ converge to a MNE of the unregularized objective $E_0$ defined in \eqref{eq:mne}?}

\subsection*{Summary of contributions.} In this work we first address \textbf{Question 1} by providing the first convergence analysis of the two-scale continuous-time Mean-Field GDA dynamics \eqref{eq:WGDA} with a \textbf{finite} scale ratio. This improves substantially the earlier convergence results by \cite{ma2021provably,domingo2022simultaneous} on Mean-Field GDA in the quasistatic setting where the scale ratio either vanishes or explodes. The key ingredient of the proof is the construction of a novel Lyapunov function which decreases exponentially along the dynamics \eqref{eq:WGDA}. We then further address \textbf{Question 2} by proving that the annealed Mean-Field GDA converges to a global MNE of the original unregularized objective $E_0$. The latter result, to the best our knowledge, is the first rigorous justification of the global convergence of GDA to MNE in the mean field regime. 

We highlight the major contributions as follows.

\begin{itemize}
    \item For any fixed  temperature $\tau>0$,  we show that in the fast ascent/descent regime (or the scale ratio $\eta$ is either larger or smaller than certain threshold), the Mean-Field GDA dynamics \eqref{eq:WGDA} converges exponentially to the MNE of the entropy-regularized objective $E_\tau$  with respect to certain Lyapunov functions; see Theorem \ref{thm:main1}. 
    
    \item The convergence in Lyapunov functions further implies the global exponential convergence of the dynamics with respect to the relative entropy (see Theorem \ref{thm:main2}).  Our (non-asymptotic) convergence results hold for any positive temperature (regularization parameter) and do not assume convexity or concavity of the interaction potential $K$. The convergence rate is characterized explicitly in terms of $\tau$, bounds on $K$ and diameters of the state spaces.

        \item We also study the simulated annealing of the Mean-Field GDA dynamics \eqref{eq:WGDA}. We prove that with the cooling schedule $\tau_t$ decaying with a logarithimc rate, the Mean-Field GDA dynamics \eqref{eq:WGDA} converges to the global mixed Nash equilibrium of $E_0$ with respect to the Nikaid\` o-Isoda error (defined in \eqref{eq:NIerror}). See Theorem \ref{thm:sa1} for the precise statements. 
\end{itemize}

\subsection{Related work.}

\subsubsection*{\textbf{Minmax optimization}} Most previous works about minmax optimization  focused on the analysis of discrete-time optimization schemes in finite dimensional Euclidean spaces under various assumptions on the objective function. In the convex-concave setting, global convergence guarantees  were obtained for various optimization schemes, such as GDA \cite{daskalakis2019last}, mirror descent \cite{mertikopoulos2018optimistic} and Hamiltonian gradient descent \cite{abernethy2019last}. In the non-convex and/or non-concave setting, local convergence to several notions of stationary points were studied in \cite{daskalakis2018limit,lin2020gradient}. In the special case where the objective satisfies a two-sided Polyak-\L ojasiewicz (PL) condition,  convergence to global Nash equilibria  were recently obtained in \cite{yang2020global,yang2022faster,doan2022convergence} for two-scale GDA. 

\subsubsection*{\textbf{Mean-field analysis}} 
Our study of the Mean-Field GDA dynamics for minmax optimization is strongly motivated by a growing amount of work on mean-field analysis of gradient descent algorithms for minimizing neural networks. The study of the latter originates from the work on two-layer networks \cite{mei2018mean,sirignano2020mean,rotskoff2022trainability,chizat2018global}, to residual networks \cite{lu2020mean,wojtowytsch2020banach} and general deep neural networks \cite{araujo2019mean,sirignano2022mean,nguyen2019mean}. These results show that the mean field dynamics can be realized as Wasserstein gradient flows of certain nonlinear energy functionals on the space of probability measures. Global convergence of the Wasserstein gradient flows were obtained for the mean-field Langevin dynamics in \cite{mei2018mean,chizat2022mean,nitanda2022convex,hu2021mean} and for the noiseless mean-field dynamics in \cite{chizat2018global}. We note that mean-field Langevin dynamics can be viewed as a special Mckean-Vlasov dynamics whose long-time convergence are often established by assuming either the small interaction strength or large noise level; see e.g. \cite{eberle2019quantitative,hu2021mean}. 

Among the existing work on mean field Langevin dynamics, \cite{chizat2022mean} is closest to ours where the author proved the exponential convergence of Mean-Field Langevin dynamics under a certain uniform log-Sobolev inequality assumption, and also proved the  convergence of the annealed Mean-Field Langevin to the  global minimizers of the objective function. Our results are parallel to the results of  \cite{chizat2022mean}, but require new proof strategies. A key ingredient of our proof is constructing new Lyapunov functions that decay exponentially along the Mean-Field GDA \eqref{eq:WGDA}. 

In the context of minmax optimization, \cite{domingo2020mean} first proposed and analyzed the noisy gradient descent ascent flow \eqref{eq:gdaparticle} without scale separation ($\eta = 1$) for finding the MNE. Specifically, the authors therein proved the convergence of \eqref{eq:gdaparticle} to the mean field GDA \eqref{eq:WGDA} in the limit of large agent size, and characterized the equilibrium of  \eqref{eq:WGDA} as the unique solution of the fixed point equation \eqref{eq:fpt}. However, \cite{domingo2020mean} did not provide a proof of the long-time convergence of  \eqref{eq:WGDA}. In \cite{ma2021provably}, the authors obtained convergence guarantees for the mean field GDA \eqref{eq:WGDA} in the quasi-static regime where the ascent dynamics is infinitely faster or slower than the descent dynamics (i.e. $\eta = +\infty \text{ or } 0$). One of the major contributions of the present work goes beyond the quasi-static regime and provides the first proof for exponential convergence of \eqref{eq:WGDA} with a finite scale ratio. 

\subsubsection*{\textbf{Wasserstein-Fisher-Rao gradient flows}}

As an alternative to Wasserstein gradient flows, gradient flow with respect to the Wasserstein-Fisher-Rao (WFR) metric \cite{chizat2018interpolating,kondratyev2016new,laschos2019geometric} has recently sparked a large amount of research on PDEs \cite{brenier2020optimal,kondratyev2019spherical}, neural network training \cite{rotskoff2019global} and statistical sampling \cite{lu2019accelerating,lu2022birth}. Especially, WFR gradient flows give rise to birth-death interacting particle dynamics that enables ``teleportation'' of mass and locations of particles and hence lead to better convergence properties than the Langevin dynamics in certain scenarios \cite{rotskoff2019global,lu2019accelerating,lu2022birth}.  More recently, \cite{domingo2020mean}  adopts a WFR-gradient flow for two-player zero-mean continuous games and proves that the time-averaging of the WFR-dynamics converges to the MNE in the regime where Fisher-Rao dynamics dominates the Wasserstein counterpart. In \cite{wang2022exponentially}, a new particle algorithm motivated by the implicit time-discretization of WFR-gradient flow was proposed for finding the MNE in the space of atomic measures. The authors proved a local exponential convergence of the proposed algorithm under certain non-degeneracy assumptions on the game objective and the  assumption that the MNE is unique.

\subsection{Notations.} For a measurable space $\mX$, we use $\mP(\mX)$ to denote the space of probability measures on $\mX$. Given a canonical Borel measure $dx$ on $\mX$, such as the uniform measure considered in the paper, if $\mu$ is absolutely continuous with respect to $dx$, we denote the Radon-Nikodym derivative by $\frac{d\mu}{dx}$. We define the shannon entropy of $\mu$ by $\mH(\mu) = -\int_{\mX} \log(\frac{d\mu}{dx}) d\mu(x)$. The relative entropy (or Kullback-Leibler divergence) and the relative Fisher information between two probability measures $\mu_1$ and $\mu_2$ are defined by 
$$
\mH(\mu_1|\mu_2) = \int \log\Big(\frac{d\mu_1}{d\mu_2}\Big) d\mu_1, \quad \mI(\mu_1|\mu_2) = \int \Big|\nabla \log \Big(\frac{d\mu_1}{d\mu_2}\Big) \Big|^2 d\mu_1
$$
if $\mu_1$ is absolutely continuous with respect to $\mu_2$ and $+\infty$ otherwise.  Given an approximate mixed Nash equilibrium $(\mu,\nu)$, we  define the Nikaid\` o-Isoda \cite{nikaido1955note} error $\NI(\mu,\nu)$ by 
\begin{equation}\label{eq:NIerror}
    \begin{aligned}
 \NI(\mu,\nu) & := \sup_{\nu^\prime \in \mP(\mY)} E_0(\mu, \nu^\prime) - \inf_{\mu^\prime \in \mP(\mX)} E_0(\mu^\prime,\nu)\\
& = E_0(\mu, \nu) - \inf_{\mu^\prime \in \mP(\mX)} E_0(\mu^\prime,\nu) +  \sup_{\nu^\prime \in \mP(\mY)} E_0(\mu, \nu^\prime)
-E_0(\mu, \nu).
\end{aligned}
\end{equation}
Note that by definition $\NI(\mu,\nu) \geq 0$ for any $\mu\in \mP(\mX)$ and $\nu\in \mP(\mY)$ and $\NI(\mu,\nu) = 0$ if and only if $(\mu,\nu)  = (\mu^\ast,\nu^\ast)$.  

\subsection{Organization} The remaining of the paper is organized as follows. In Section \ref{sec:main} we state the key assumptions and present the main results of the paper. In Section \ref{sec:application} we discuss  applications of our results in training GANs and adversarial learning of PDEs. Section \ref{sec:proofs} devotes to the proofs of main results. We finish with several conclusion remarks and future directions in Section  \ref{sec:conclusion}. Further proof details are provided in the Appendix.

\section{Assumptions and Main Results} \label{sec:main}

Let us first recall the minmax problem of the entropy-regularized continuous game:
\begin{equation}
\label{eq:minmax2}
\min_{\mu\in \mP(\mX)} \max_{\nu\in \mP(\mathcal{Y})} E_{\tau}(\mu, \nu), \text{ where } E_{\tau}(\mu,\nu) := \int_{\mX} \int_{\mathcal{Y}} K(x,y) d\nu(y)d\mu(x) - \tau \mH(\mu) + \tau \mH(\nu).
\end{equation}
When the regularization parameter (or temperature) $\tau = 0$, we write 
 $$
 E_0(\mu,\nu) = \int_{\mX} \int_{\mathcal{Y}} K(x,y) d\nu(y)d\mu(x).
 $$
We are interested in the global convergence of the Mean-Field GDA dynamics to the MNE of \eqref{eq:minmax2}:
$$
\begin{aligned}
\partial_t \mu_t & = \nabla \cdot (\mu_t \int_{\mY} \nabla_x K(x,y)d\nu_t(y)) +\tau \Delta \mu_t,\\
\partial_t \nu_t & = \eta\Big(-\nabla \cdot (\nu_t \int_{\mX} \nabla_y K(x,y)d\mu_t(x)) + \tau\Delta \nu_t\Big).
\end{aligned}
$$

 In what follows,  we may suppress the dependence of $E_{\tau}$ on $\tau$ and write $E= E_\tau$ to simplify notations when the $\tau>0$ is fixed; we will indicate such dependence in later discussions on the annealed dynamics where $\tau = \tau_t$ shrinks to zero in time. 

\subsection{Assumptions}
Following \cite{domingo2022simultaneous} and \cite{wang2022exponentially}, we make the following assumptions on the state spaces $\mX$ and $\mY$.

\begin{assum}\label{assS}
   The state spaces $\mX$ and $\mY$ are (i) either smooth compact manifolds without boundary such that the eigenvalues of the Ricci curvature are strictly positive everywhere or (ii) Euclidean tori (with possibly unequal dimensions). 
\end{assum}

We also make the following smoothness assumption on the potential $K$. 

\begin{assum}\label{assK}
The function $K\in C^2(\mX\times \mY)$ and there exists $K_{xy}> 0$ such that 
$$
 \|\nabla_{xy}^2 K\|_{\infty} \leq K_{xy}.
 $$
\end{assum}
It will be very useful to introduce operators $\mK^+ : \mP(\mX) \gt \mP(\mY)$ and $\mK^- : \mP(\mY) \gt \mP(\mX)$ where 
\begin{equation}\label{eq:mK}
    \mK^+ \mu (dy) = \frac{1}{Z^+(\mu)} \exp\Big(\int_{\mX} \tau^{-1} K(x,y)d\mu(x) \Big), \quad \mK^- \nu (dx) = \frac{1}{Z^-(\nu)} \exp\Big(-\int_{\mY} \tau^{-1} K(x,y)d\nu(y) \Big).
\end{equation}
Here $Z^-(\nu)$ and $Z^+(\mu)$ are the corresponding partition functions defined by 
\begin{equation} \label{eq:Z} 
Z^-(\nu) = \int_{\mX} \exp\Big( - \int_{\mY} \tau^{-1} K(x,y)d\nu (y)\Big) dx, \quad Z^+(\mu)= \int_{\mY} \exp\Big( \int_{\mX} \tau^{-1} K(x,y)d\mu (x)\Big) dy.
\end{equation} 
It is easy to verify by the definition of $E(\mu,\nu)$ that 
\begin{equation}\label{eq:mK2}
   \argmax_{\nu\in \mP(\mY)} E(\mu,\nu) =  \mK^+(\mu), \quad \argmin_{\mu\in \mP(\mX)} E(\mu, \nu) = \mK^{-}(\nu). 
\end{equation}
With the notations above, the MNE $(\mu^\ast, \nu^\ast)$ of \eqref{eq:minmax2} is characterized as the unique solution of the following fixed point equations (see \cite[Theorem 1]{domingo2020mean}):
\begin{equation}\label{eq:fixedpoint}
   \mu^\ast = \mK^-(\nu^\ast), \quad \nu^\ast = \mK^+(\mu^\ast). 
\end{equation}

The lemma below shows that under Assumption \ref{assK} and Assumption \ref{assS}, the measures $\mK^{+} (\mu)$ and $\mK^-(\nu)$ satisfy the logarithmic Sobolev inequalities (LSI) uniformly for any $\mu \in \mP(\mX)$ and $\nu \in \mP(\mY)$.

\begin{lemma}\label{lem:LSI}
There exists a constant $\lambda_{LS}>0$  such that for any $\mu \in \mP(\mX)$ and $\nu \in \mP(\mY)$, 
\begin{equation}\label{eq:LSI}
    \lambda_{LS} \mH(\mu | \mK^-{\nu}) \leq \mI(\mu| \mK^-{\nu}), \quad \lambda_{LS} \mH(\nu| \mK^+{\mu}) \leq \mI(\nu| \mK^+{\mu}).
\end{equation}
\end{lemma}
Lemma \ref{lem:LSI} follows from the classical Holley-Stroock argument and the fact that the uniform measures on $\mX$ and $\mY$ satisfy the LSI thanks to Assumption \ref{assS}; see e.g. \cite[Proposition 5.2]{chizat2022mean} and   \cite[Proposition 6]{domingo2022simultaneous}). It is worthwhile to note that in the regime where $\tau^{-1} \|K\|_{\infty} \gg 1$, the log-Sobolev constant $\lambda_{LS} = \lambda_{LS}(\tau)$ satisfies that
\begin{equation}\label{eq:bdLSI}
    \lambda_{LS}(\tau) = \mathcal{O}(\exp(-\tau^{-1} \|K\|_{\infty})). 
\end{equation}

\subsection{Exponential convergence of Mean-Field GDA with a fixed temperature}

The goal of this section is to present the global convergence of the GDA dynamics \eqref{eq:WGDA} to the MNE $(\mu^\ast, \nu^\ast)$.
To this end, it will be useful to define the following functionals 
\begin{equation}
\begin{aligned}\label{eq:mLs}
        \mL_1 (\mu) & := \max_{\nu\in \mP(\mY)} E(\mu, \nu) - \min_{\mu\in \mP(\mX)}\max_{\nu\in \mP(\mY)} E(\mu, \nu),\\
        \mL_2 (\mu,\nu) & := \max_{\nu\in \mP(\mY)}  E(\mu, \nu) - E(\mu, \nu),\\
        \mL_3 (\nu) & := \max_{\nu\in \mP(\mY)} \min_{\mu\in \mP(\mX)} E(\mu, \nu)- \min_{\mu\in \mP(\mX)} E(\mu, \nu),\\
        \mL_4 (\mu,\nu) & :=   E(\mu, \nu)- \min_{\mu\in \mP(\mX)}  E(\mu, \nu).
\end{aligned}
\end{equation}
Note that by definition the functionals $\mL_i,i=1,\cdots,4$  defined above are non-negative. 
For a fixed constant $\gamma >0$, we also define the Lyapunov functions
\begin{equation}\label{eq:mLtilde}
 \mL (\mu, \nu) = \mL_1 (\mu) + \gamma \mL_2 (\mu,\nu), \quad \widetilde{\mL} (\mu, \nu) = \mL_3 (\mu) + \gamma \mL_4 (\mu,\nu).   
\end{equation}
Let $\kappa$ be the {\em effective} condition number  defined by 
$$
\kappa = \frac{K_{xy}}{\tau \lambda_{LS}}.
$$
Our first main theorem characterizes the exponential convergence of \eqref{eq:WGDA} in terms of the Lyapunov functions $\mL$ and $\widetilde{\mL}$. 
\begin{theorem}\label{thm:main1} Let Assumption \ref{assS}  and Assumption \ref{assK} hold. For a fixed $\gamma < 1$, let $(\mu_t,\nu_t)$ be the solution to the GDA dynamics \eqref{eq:WGDA} with an initial condition $(\mu_0,\nu_0)$ satisfying $\mL(\mu_0,\nu_0) <\infty$ and $\widetilde{\mL}(\mu_0,\nu_0) <\infty$, and with a finite time-scale ratio $\eta>0$. 

(i) \textbf{Fast ascent regime: } set $
\eta = \frac{2\lambda_{LS}\kappa^2\diam(\mY)^2}{\gamma}.
$
Then for all $t>0$,
$$
\mL (\mu_t, \nu_t) \leq e^{-\alpha_1 t}  \mL (\mu_0, \nu_0)
$$
with 
$$\alpha_1 =  \tau\lambda_{LS} \Big(\frac{1-\gamma}{2} \wedge  \frac{\kappa^2 \diam(\mY)^2(1+3\gamma)}{\gamma}\Big);
$$

(ii) \textbf{Fast descent regime: } set $
\eta  = \frac{\gamma}{2\lambda_{LS}\kappa^2\diam(\mX)^2 (1+ 3\gamma)}.
$
Then for all $t>0$,
$$
\widetilde{\mL} (\mu_t, \nu_t) \leq e^{-\alpha_2 t}  \widetilde{\mL} (\mu_0, \nu_0)
$$
with 
$$\alpha_2 = \frac{\tau \lambda_{LS}}{2} \Big(1\wedge  \eta(1-\gamma)\Big).
$$
\end{theorem}

\begin{remark}
Theorem \ref{thm:main1} states that the two-scale GDA dynamics \eqref{eq:WGDA} with a suitable finite scale ratio converges exponentially to the equilibrium. This improves substantially the earlier result by \cite{ma2021provably} on GDA  in the quasistatic setting where the scale ratio $\eta = 0$ or $\eta = \infty$. We emphasize that we chose the specific scale ratios merely for the purpose of simplifying the expression of the convergence rate $\alpha$. In fact, by tracking the proof of Theorem \ref{thm:main1}, one can obtain that  the convergence rate 
$$\begin{aligned}
    \alpha = \begin{cases}
C_1 \tau \lambda_{LS},  & \text{ if } \eta > c_1 \lambda_{LS}\kappa^2\diam(\mY)^2,\\
C_2 \tau \lambda_{LS}^2, & \text{ if } \eta <  \frac{c_2\gamma}{\lambda_{LS}\kappa^2\diam(\mY)^2(1+3\gamma)},
\end{cases}
\end{aligned}
$$
for some  constants $C_i, c_i>0, i=1,2$. Moreover, the dependence of the convergence rates on the diameters of the state spaces via $K_{xy} \diam(\mY)$ and $K_{xy}\diam(\mX)$  can be avoided, and especially the latter two quantities can be replaced by $\|\nabla K\|_{\infty}$; see discussions around \eqref{eq:reskl2}.

Observe also that in the low temperature regime ($\tau \ll 1$), the log-Sobolev constant $\lambda_{LS} = \mathcal{O}(\exp(-\tau^{-1} \|K\|_{\infty}))$. Hence Theorem \ref{thm:main1} requires $\eta = \Omega((\tau^{-1} K_{xy})^2/\lambda_{LS})$ (and $\eta = o(\lambda_{LS}\tau^2/ K_{xy}^2)$) to guarantee an exponential convergence rate $\alpha = \mathcal{O}(\tau \lambda_{LS})$ (and $\alpha = \mathcal{O}(\lambda_{LS}^2\tau^3)$) in the fast ascent (descent) regime. Notice that the quadratic dependence of the convergence rate $\alpha$ on $\lambda_{LS}$ in the fast descent regime, in contrast to the linear dependence on $\lambda_{LS}$ in the fast ascent regime, is due to the fact that the ascent dynamics itself is running in the slow time-scale $\eta = o(\lambda_{LS}\tau^2/ K_{xy}^2)$; one would recover the same rate of convergence as in the fast ascent regime if one used the time scales $1/\eta$ for the descent dynamics and $1$ the ascent dynamics.  
\end{remark}

\begin{remark}
    Our construction of Lyapunov functions \eqref{eq:mLtilde} is strongly motivated by the recent study \cite{yang2020global, doan2022convergence} of two-scale GDA for minmax optimization on Euclidean spaces. In the finite dimensional setting, a two-sided PL condition is sufficient to guarantee the exponential convergence of two-scale GDA in both continuous-time \cite{doan2022convergence} and discrete-time \cite{yang2020global,yang2022faster} cases. In our infinite dimensional setting of the Mean-Field GDA, the uniform log-Sobolev inequality plays the same role as the two-sided PL condition, which however, needs to be combined with the entropy sandwich inequalities in Lemma \ref{lem:bdL} to obtain the exponential dissipation of Lyapunov functions.  
\end{remark}

Our next theorem provides an exponential convergence of \eqref{eq:WGDA} with respect to the relative entropy.
\begin{theorem} \label{thm:main2} Let the assumptions of Theorem \ref{thm:main1} hold. Let the rates $\alpha_i,i=1,2$ be defined as in Theorem \ref{thm:main1}. 

(i) \textbf{Fast ascent regime:} set $
\eta = \frac{2\lambda_{LS}\kappa^2\diam(\mY)^2}{\gamma}.
$
Then for all $t>0$,
\begin{equation}\label{eq:tauHnut}
\tau\mH(\mu_t | \mu^\ast)  \leq \frac{e^{-\alpha_1 t}}{\gamma} \mL(\mu_0,\nu_0) \text{ and }  \tau\mH(\nu_t | \nu^\ast) \leq \Big(\frac{2}{\gamma} + \frac{4\|K\|_{\infty}^2}{\tau^2}\Big) e^{-\alpha_1  t}\mL(\mu_0,\nu_0).
\end{equation}

(ii) \textbf{Fast descent regime:} set $
\eta = \frac{2\lambda_{LS}\kappa^2\diam(\mY)^2}{\gamma}.
$
Then for all $t>0$,
\begin{equation}
\tau\mH(\mu_t | \mu^\ast)  \leq \frac{e^{-\alpha_2 t}}{\gamma} \mL(\mu_0,\nu_0) \text{ and }   \tau\mH(\nu_t | \nu^\ast) \leq \Big(\frac{2}{\gamma} + \frac{4\|K\|_{\infty}^2}{\tau^2}\Big) e^{-\alpha_2 t}\mL(\mu_0,\nu_0).
\end{equation}
\end{theorem}

\subsection{Convergence of annealed Mean-Field GDA}\label{sec:SA}
In this section, we proceed to presenting the convergence of the ``annealed'' Mean-Field GDA dynamics 
\begin{equation}
    \label{eq:anealedWGDA}
\begin{aligned}
\partial_t \mu_t & = \nabla \cdot \left(\mu_t \int_{\mY} \nabla_x K(x,y)d\nu_t(y)\right) +\tau_t \Delta \mu_t,\\
\partial_t \nu_t & = \eta_t\left(-\nabla \cdot \left(\nu_t \int_{\mX} \nabla_y K(x,y)d\mu_t(x)\right) + \tau_t \Delta \nu_t\right),
\end{aligned}
\end{equation}
where $\tau_t>0$ is now a time-dependent temperature which shrinks in time. Given any initial condition $(\mu_0,\nu_0) \in \mP(\mX)\times \mP(\mY)$, the  existence of uniqueness of the global solution to \eqref{eq:anealedWGDA} follow directly from the classical  well-posedness of the theory of nonlinear Mckean-Vlasov-type PDEs \cite{funaki1984certain,sznitman1991topics}.  Our goal is to show that by carefully choosing  the cooling schedule $\tau_t$ and the time-scale ratio $\eta_t$ the solution  $(\mu_t,\nu_t)$ to the annealed dynamics \eqref{eq:anealedWGDA}  converges to the MNE of $E_0$.

 Let $(\mu_{\tau}^\ast, \nu_{\tau}^\ast)$ be the solution of \eqref{eq:fixedpoint} corresponding to temperature $\tau$. Recall the Nikaid\` o-Isoda defined by \eqref{eq:NIerror}.

\begin{theorem}\label{thm:sa1}
Let Assumption \ref{assK} and Assumption \ref{assS} hold. Let $(\mu_t,\nu_t)$ be the solution to \eqref{eq:anealedWGDA} with an initial condition $(\mu_0,\nu_0) \in \mP(\mX)\times \mP(\mY)$.  Assume that the log-Sobolev constant $\lambda_{LS} = \lambda_{LS}(\tau) \geq C_0 e^{-\xi^\ast/\tau}$ for some $\xi^\ast,C_0>0$.

\textbf{(i) Fast ascent regime:}  Assume further that $\tau_t$ is smooth, decreasing in $t$ and for some $\xi>\xi^\ast$, $\tau_t  = \xi/\log t$ for large values of $t$.  Set $
\eta_t =  \frac{M }{(\log t)^2}t^{\xi^\ast/\xi}
$ for some large $M>0$. 
Then for every $0<\epsilon < 1-\xi^\ast/\xi$, there exists $C, C^\prime>0$ such that for $t$ sufficiently large,  
\begin{equation}\label{eq:Lmutstar}
\mH(\mu_t | \mu_{\tau_t}^\ast) \leq C t^{-( 1 - \xi^\ast/\xi-\epsilon  )} , \qquad \mH(\nu_t | \nu_{\tau_t}^\ast) \leq C t^{-( 1 - \xi^\ast/\xi -\epsilon )} ,
\end{equation}
and that 
\begin{equation}\label{eq:NImut}
    0\leq  \NI(\mu_t, \nu_t)  \leq  \frac{C^\prime \log \log t}{\log t}. 
\end{equation}

\textbf{(ii) Fast descent regime:} Assume further that $\tau_t$ is smooth, decreasing in $t$ and for some $\xi>2\xi^\ast$, $\tau_t  = \xi/\log t$ for large values of $t$.  Set  $\eta_t = \frac{\log t}{M t}$ for some large M. Then for every $0<\epsilon < 1-2\xi^\ast/\xi$, there exists $C, C^\prime>0$ such that for $t$ sufficiently large,  
\begin{equation}\label{eq:Lmutstar2}
\mH(\mu_t | \mu_{\tau_t}^\ast) \leq C t^{-( 1 -  2\xi^\ast/\xi -\epsilon)}, \qquad \mH(\nu_t | \nu_{\tau_t}^\ast) \leq C t^{-( 1 -2 \xi^\ast/\xi-\epsilon )},
\end{equation}
and that 
\begin{equation}\label{eq:NImut2}
    0\leq  \NI(\mu_t, \nu_t)  \leq  \frac{C^\prime\log \log t}{\log t}. 
\end{equation}
\end{theorem}

\section{Applications}\label{sec:application}
In this section, we discuss briefly applications of our theoretical results in  training of GANs and adversarial learning of PDEs. 

\subsection{Training of GANs}\label{sec:gan}
 Let $\mu_m$ be the empirical measure associated to the  i.i.d. samples $\{x_i\}_{i=1}^m \in \mX$ from a target measure $\mu\in \mP(\mX)$. Let $D_{\mF}$ be an IPM on the space of probability measures $\mP(\mX)$ parameterized by a set $\mF$ of discriminators, i.e. 
$$
D_{\mF}(\mu, \nu) := \sup_{f\in \mF} \int f d\mu - \int fd\nu. 
$$
IPM-based GAN learns an optimal $\mu$ that minimizes  $D_{\mF} (\mu, \mu_m)$ over $\mP(\mX)$ :
\begin{equation}\label{eq:minmaxgan}
   \inf_{\mu \in \mP(\mX)} D_{\mF}(\mu, \mu_m) =  \inf_{\mu \in \mP(\mX)} \sup_{f\in \mF} \int f d\mu - \int fd\mu_m.  
\end{equation}
Consider the witness function class $\mF$ given by the unit ball of Barron space $\mathcal{B}$ which consists of functions admitting the representation 
$$
f(x) =  \int_{\mY} a \sigma(b\cdot x + c)d\nu(y),
$$
where $y = (a,b,c) \in \mY$ and $\nu$ is a probability measure on the parameter space $\mY$. Observe that Barron functions arise as natural infinite-width limit of two-layer neural networks with a  dimension-free rate  \cite{bach2017breaking,ma2022barron,barron1993universal}.  When the activation function satisfies that $\sup_{x} |a \sigma(b\cdot x+c)| \leq \phi(y)$ for some nonnegative function $\phi$ and for all $y\in \mY$, the Barron norm $\|f\|_{\mB}$ is defined by 
\begin{equation} \label{eq:barronnorm}
\|f\|_{\mB} := \inf_{\nu} \Big\{ \int_{\mY} \phi(y)\mu(dy) \ \Big|\ 
f(x) =  \int_{\mY} a \sigma(b\cdot x + c)d\nu(y) \Big\}.
\end{equation}
Setting $\mF = \{f\in \mB \ |\ \|f\|_{\mB} \leq 1 \}$ in \eqref{eq:minmaxgan} leads to 
\begin{equation} \label{eq:ganobj}
    \inf_{\mu\in \mP(\mX)}  \sup_{\nu\in \mP(\mY)} \int_{\mX}\int_{\mY}K(x,y) \mu(dx) \nu(dy), \text{ where } K(x,y) = \Sigma(x,y) - \int_{\mX} \Sigma(x,y) \mu_m(dx)+ \phi(y).
\end{equation}
Here we adopted the short-notation $\Sigma(x,y) = a \sigma(b\cdot x + c)$ with $y = (a,b,c) \in \mY$  in the above. 
Assume that the activation function $\sigma \in C^2(\R)$, and the input space $\mX$ and the parameter space satisfy the Assumption \ref{assS}. Then it is straightforward to see that  $K\in C^2(\mX \times \mY)$ and there exists $C_\sigma < \infty$ such that for any multi-indices $\mathbf{i}$ and $\mathbf{j}$ with $0\leq |\mathbf{i}|+|\mathbf{j}| \leq 2$,  
$$
\|\nabla_{x}^\mathbf{i} \nabla^\mathbf{j}_y K(x,y)\|_{\infty} \leq 2\|\nabla_{x}^\mathbf{i} \nabla^\mathbf{j}_y \Sigma\|_{\infty} + \|\nabla^{\mathbf{j}}_y \phi\| \leq C_{\sigma}.
$$
Therefore the convergence results in Theorem \ref{thm:main1}-\ref{thm:main2} for the Mean-Field GDA hold for the entropy-regularization of the GAN objective defined in \eqref{eq:ganobj}. Moreover, Theorem \ref{thm:sa1} implies that the annealed GDA dynamics finds the MNE of the unregularized GAN objective. 

\subsection{Adversarial learning of PDEs} We provide another usage of our results in adversarial learning of PDEs. To demonstrate the idea, we focus on a simple linear elliptic PDE on a bounded Lipschitz domain $\mZ\subset \R^d$ equipped with the Neumann boundary condition
$$
\begin{aligned}
    -\Delta u(z)  + Vu(z) & = f(z), z\in \mZ,\\
    \partial_{\nu} u(z) & = 0,  z\in \partial \mZ.
\end{aligned}
$$
Assume that $0< V_{\min} \leq V\leq  V_{\max} < \infty$ and $f\in (H^1(\mZ))^{*}$. The weak solution $u \in H^1(\mZ)$ satisfies that 
\begin{equation}\label{eq:weakPDE}
    \int_{\mZ} \nabla u \cdot \nabla v+  V u v dz =  \int_{\mZ} f v dz, \forall v\in H^1(\mZ).
\end{equation}
We seek an approximate solution to \eqref{eq:weakPDE} in the framework of Petrov-Galerkin \cite{Petrov40,mitchell1980finite} where we choose the spaces of trial functions and test functions as two different Barron functions. More precisely, consider a trial function $u(z)  \in \mU  := \{u\in \mB_1\ |\ \|u\|_{\mB_1} \leq 1\}$ and a test function $v \in \mV  := \{v\in \mB_2\ |\ \|v\|_{\mB_2} \leq 1\}$, where $\mB_i,i=1,2$ are Barron spaces defined in Section \ref{sec:gan} with activation function $\sigma_i$ and  Barron norm $\|\cdot\|_{\mB_i}$ defined in \eqref{eq:barronnorm} with $\phi$ replaced by nonnegative weight functions $\phi_i$. We look for a solution $u \in \mU$ parameterized by some probability measure $\mu \in \mP(\mX)$,
$$
u(z) = \int_{\mX} a_1\sigma_1(b_1\cdot z + c_1) \mu(dx)
$$
with $x = (a_1,b_1,c_1) \in \mX$ satisfying equation \eqref{eq:weakPDE}  for any $v\in \mV$ with $\|v\|_{\mB_2} \leq 1$ parameterized by $\nu\in \mP(\mY)$ such that 
$$
v(z) = \int_{\mY} a_2\sigma_2(b_2\cdot z + c_2) \nu(dy).
$$
Putting these into the weak formulation \eqref{eq:weakPDE} leads to 
$$
\inf_{\mu\in \mP(\mX)} \sup_{\nu\in \mP(\mY)} \int_{\mX} \int_{\mY} K(x,y) \mu(dx)\nu(dy),
$$
where the potential $K(x,y)$ is given for $x=(a_1,b_1,c_1), y=(a_2,b_2,c_2)$ by 
$$\begin{aligned}
    K(x,y)& = \int_{\mZ} \Big(a_1 a_2 b_1 \cdot b_2  \sigma^\prime_1(b_1\cdot z+ c_1)\sigma^\prime_2(b_2\cdot z+ c_2) + V(z)a_1 a_2\sigma_1(b_1\cdot z + c_1) \sigma_2(b_2\cdot z + c_2)\\
    & \qquad -f(z) a_2\sigma_2(b_2\cdot z + c_2)\Big) dz - \phi_1(x) + \phi_2(y).
\end{aligned}
$$
Assume that the activation functions $\sigma_i \in C^2(\R)$ and the parameter spaces $\mX$ and $\mY$ satisfy  Assumption \ref{assS}. Assume also that $\phi_1\in C^2(\mX)$ and $\phi_2 \in C^2(\mY)$. Then it is easy to verify that $K\in C^2(\mX\times \mY)$ and that $
\|K\|_{C^2}\leq C
$ for some constant $C>0$ depending on $\sigma_i, \phi_i, \mX, \mY, V$ and $f$. Hence the convergence results on Mean-Field GDA and its annealed version established in Section \ref{sec:main} apply to this problem.  

\section{Proofs of Main Results} \label{sec:proofs}

\subsection{Proof of convergence for Mean-Field GDA with a fixed temperature}
We only present the proof of our main convergence results in the fast ascent regime. The proofs for the fast descent regime can be carried out in a similar manner and are provided in the Appendix \ref{sec:appmfgda} for completeness.

We first state a lemma below summarizing some  important properties on the functionals $\mL_i,i=1,\cdots,4$. 
defined in  \eqref{eq:mLs}.

\begin{lemma}\label{lem:bdL}
For any $\mu\in \mP(\mX)$ and $\nu \in \mP(\mY)$, the following hold 

\begin{align}\label{eq:ml2}
\mL_2 (\mu,\nu) & =\tau \mH(\nu | \mK^+ (\mu)),\\
\label{eq:ml4}
\mL_4 (\mu,\nu) & =\tau \mH(\mu | \mK^- (\nu)),\\
\label{eq:ml1}
\tau\mH(\mu | \mu^\ast) \leq \mL_1 (\mu) & \leq \tau\mH(\mu | \mK^-(\mK^+ (\mu))),\\
\label{eq:ml3}
\tau\mH(\nu | \nu^\ast)\leq \mL_3 (\nu) & \leq \tau\mH(\nu | \mK^+(\mK^- (\nu))).
\end{align}
\end{lemma}
\begin{remark}
The sandwich inequalities \eqref{eq:ml1} and \eqref{eq:ml3} play an essential role in  controlling terms in the entropy production of the Mean-Field GDA dynamics via the Lyapunov functions in order to close the Gr\"onwall argument to obtain the dissipation of the latter along the dynamics. A similar sandwich inequality  appeared in \cite[Lemma 3.4]{chizat2022mean} in the proof of convergence for the  Mean-Field Langevin dynamics. 
\end{remark}

Next, we keep track of the time-derivatives of $\mL_1(\mu_t)$ and $\mL_2(\mu_t,\nu_t)$ in the next proposition whose proof can be found in Appendix \ref{sec:applem}.

\begin{proposition}\label{prop:diffL}
Let $(\mu_t,\nu_t)$ be the solution to the DGA dynamics \eqref{eq:WGDA}. Then 
\begin{equation}
    \begin{aligned}\label{eq:dtL1}
        \frac{d}{dt} \mL_1(\mu_t) \leq -\frac{\tau^2}{2} \mI(\mu_t | \mK^-(\mK^+(\mu_t))) + K_{xy}^2  \diam (\mY)^2\cdot \mH(\nu_t | \mK^+(\mu_t))
    \end{aligned}
\end{equation}
and 
\begin{equation}
    \begin{aligned}\label{eq:dtL2}
        \frac{d}{dt} \mL_2(\mu_t, \nu_t) \leq -\tau^2\eta\mI(\nu_t | \mK^+(\mu_t)) + \frac{\tau^2}{2} \mI(\mu_t | \mK^-(\mK^+(\mu_t))) +   3K_{xy}^2  \diam (\mY)^2 \cdot \mH(\nu_t | \mK^+(\mu_t)).
         \end{aligned}
\end{equation}
\end{proposition}

With Proposition \ref{prop:diffL}, we are ready to present the proof of Theorem \ref{thm:main1} and Theorem \ref{thm:main2}.

\begin{proof}[Proof of Theorem \ref{thm:main1}]
Thanks to Proposition \ref{prop:diffL} and identity \eqref{eq:ml2}, we have
$$
\begin{aligned}
    \frac{d }{dt}\mL (\mu_t, \nu_t) & \leq -\eta\tau^2 \gamma\mI(\nu_t | \mK^+(\mu_t)) -\frac{\tau^2}{2} (1 - \gamma)  \mI(\mu_t | \mK^-(\mK^+(\mu_t)))\\
    & \qquad + \frac{K_{xy}^2 \diam(\mY)^2(1+3\gamma)}{\tau }\mL_2(\mu_t,\nu_t).
\end{aligned}
$$
Observe also from Lemma \ref{lem:LSI} and sandwich inequality \eqref{eq:ml1} that 
$$
\begin{aligned}
    \tau\mI(\mu_t | \mK^-(\mK^+(\mu_t))) & \geq \tau \lambda_{LS} \mH(\mu_t |\mK^-(\mK^+(\mu_t)) )\\
    & \geq  \lambda_{LS} \mL_1(\mu_t).
\end{aligned}
$$
Combining the last two displays leads to 
$$
\begin{aligned}
    \frac{d }{dt}\mL (\mu_t, \nu_t) & \leq -\eta\tau^2 \gamma\mI(\nu_t | \mK^+(\mu_t)) -\frac{\tau}{2} (1 - \gamma) \lambda_{LS}  \mL_1(\mu_t) + \frac{K_{xy}^2 \diam(\mY)^2(1+3\gamma)}{\tau} \mL_2(\mu_t,\nu_t)\\
    & \leq -\eta\tau^2 \gamma \lambda_{LS} \mH(\nu_t | \mK^+(\mu_t)) -\frac{\tau}{2} (1 - \gamma) \lambda_{LS}  \mL_1(\mu_t) + \frac{K_{xy}^2 \diam(\mY)^2(1+3\gamma)}{\tau} \mL_2(\mu_t,\nu_t)\\
    & = -\frac{\tau}{2} (1 - \gamma) \lambda_{LS}  \mL_1(\mu_t) - \tau\Big(\eta \gamma \lambda_{LS} -\frac{K_{xy}^2 \diam(\mY)^2(1+3\gamma)}{\tau^2} \Big)\mL_2(\mu_t,\nu_t).
\end{aligned}
$$
Now for any fixed $\gamma < 1$, we set 
$$
\eta = \frac{2K_{xy}^2 \diam(\mY)^2(1+3\gamma)}{\tau^2 \gamma \lambda_{LS}} = \frac{2\lambda_{LS}\kappa^2\diam(\mY)^2}{\gamma}.
$$
Then it follows from the last inequality that 
$$
\frac{d}{dt} \mL (\mu_t, \nu_t) \leq -\alpha \mL (\mu_t, \nu_t) 
$$
with 
$$\begin{aligned}
\alpha=\tau \lambda_{LS} \Big(\frac{1-\gamma}{2} \wedge  \frac{\kappa^2 \diam(\mY)^2(1+3\gamma)}{\gamma}\Big).
\end{aligned}
$$
\end{proof}

\begin{proof}[Proof of Theorem \ref{thm:main2}]
First, thanks to Theorem \ref{thm:main1}, $\mL(\mu_t,\nu_t) \leq e^{-\alpha t} \mL(\mu_0,\nu_0) $. In particular, for any $t>0$, 
\begin{equation} \label{eq:mL1nu}
\mL_1(\mu_t) \leq e^{-\alpha t} \mL(\mu_0,\nu_0) 
\end{equation}
and 
\begin{equation} \label{eq:Hnu}
\tau\mH (\nu_t | \mK^+(\mu_t)) =  \mL_2(\mu_t, \nu_t) \leq \frac{e^{-\alpha t}}{\gamma}\mL(\mu_0,\nu_0). 
\end{equation}
As a result of \eqref{eq:mL1nu} and \eqref{eq:ml1}, one obtains that 
\begin{equation} \label{eq:Hmu}
\tau \mH (\mu_t |\mu^\ast) \leq  e^{-\alpha t} \mL(\mu_0,\nu_0).
\end{equation}

Next, to obtain the exponential decay of $\mH(\nu_t | \nu^\ast) $, notice first that 
$$\begin{aligned}
 \tau  \mH(\nu_t | \nu^\ast) &  = \tau  \mH(\nu_t | \mK^+(\mu_t)) +  \tau \int_{\mY} (\log (\mK^+(\mu_t)) - \log(\nu^\ast)) d\nu_t\\
 & = \tau  \mH(\nu_t | \mK^+(\mu_t)) +  \tau \int_{\mY} (\log (\mK^+(\mu_t)) - \log(\nu^\ast)) d(\nu_t  - \nu^\ast)  - \tau \mH(\nu^\ast | \mK^+(\nu_t))\\
 & \leq \tau  \mH(\nu_t | \mK^+(\mu_t)) +  \tau \int_{\mY} (\log (\mK^+(\mu_t)) - \log(\nu^\ast)) d(\nu_t  - \nu^\ast).
\end{aligned}
$$
Since $\nu^\ast  = \mK^+(\mu^\ast)$, we have 
$$\begin{aligned}
(\log (\mK^+(\mu_t)) - \log(\nu^\ast)) (y) & = \tau^{-1} \int_{\mX} K(x,y) (\mu_t(dx) - \mu^\ast(dx)) - (\log Z^+(\mu_t) - \log Z^+(\mu^\ast)).
\end{aligned}
$$
It follows from the last two displays that 
\begin{equation}\label{eq:Hmu2}
    \begin{aligned}
 \tau  \mH(\nu_t | \nu^\ast) &  \leq  \tau  \mH(\nu_t | \mK^+(\mu_t))
  + \int_{\mY}  \int_{\mX} K(x,y) (\mu_t(dx) - \mu^\ast(dx))  (\nu_t(dy)  - \nu^\ast(dy)). 
 \end{aligned}
\end{equation}
The last term on the right side above can be upper bounded by 
\begin{equation}
    \begin{aligned}\label{eq:bdmutmu}
\int_{\mY}  \int_{\mX} K(x,y) (\mu_t(dx) - \mu^\ast(dx))  (\nu_t(dy)  - \nu^\ast(dy)) & \leq 
\|K\|_{\infty}\cdot \text{TV} (\mu_t, \mu^\ast) \cdot \text{TV} (\nu_t, \nu^\ast) \\
& \leq  2\|K\|_{\infty} \sqrt{\mH(\mu_t| \mu^\ast)} \cdot \sqrt{\mH(\nu_t| \nu^\ast)}\\
& \leq \frac{\tau}{2} \mH(\nu_t| \nu^\ast) + \frac{2 \|K\|_{\infty}^2}{\tau} \mH(\mu_t| \mu^\ast).
 \end{aligned}
\end{equation}
 where we have used the Pinsker's inequality and Young's inequality in the last two lines above.  Finally combining the last two estimates leads to 
 $$
\begin{aligned}
   \tau  \mH(\nu_t | \nu^\ast)  & \leq 2\tau  \mH(\nu_t | \mK^+(\mu_t)) + \frac{4 \|K\|_{\infty}^2}{\tau} \mH(\mu_t| \mu^\ast)\\
   & \leq \Big(\frac{2}{\gamma} + \frac{4 \|K\|_{\infty}^2}{\tau^2}\Big)e^{-\alpha t}\mL(\mu_0,\nu_0),
\end{aligned}
 $$
 where we have used inequalities \eqref{eq:Hnu} and \eqref{eq:Hmu} in the last inequality above.

\end{proof}

\subsection{ Proof of convergence for the annealed dynamics}
Recall that the MNE of entropy-regularized objective $E_\tau$ is given by $(\mu^\ast_\tau, \nu^\ast_\tau)$ characterized by \eqref{eq:mne}.  Note that we emphasized the dependence of the objective and the optimizer on the temperature $\tau$ since the later will be time-dependent throughout this section. 

It will be useful to define energies $\mE_i, i=1,2$ as follows. 
\begin{equation}\label{eq:mK3}\begin{aligned}
    \mE_1(\mu) & :=   
\max_{\nu\in \mP(\mY)} E_\tau(\mu,\nu)  = -\tau \mH(\mu) + \tau\log Z^+(\mu),\\
    \mE_2(\nu) & := \min_{\mu\in \mP(\mX)} E_\tau(\mu, \nu)  = \tau \mH(\nu) -\tau\log Z^-(\nu).
\end{aligned}
\end{equation}
It is clear that $\mL_1(\mu) = \mE_1(\mu) - \mE_1(\mu^\ast)$.
\begin{proof}[Proof of Theorem \ref{thm:sa1}]
Similar to the last section, we only present here the proof for the fast ascent regime. The proof for the other regime can be carried out analogously which for completeness is provided in Appendix \ref{sec:appsa}.  The  proof of the entropy decays in \eqref{eq:Lmutstar} follows largely the proof of Theorem \ref{thm:main1} and is also inspired by the proof of \cite[Theorem 4.1]{chizat2022mean}. 

\textbf{Step 1.} Bounding $\mL_1(\mu_t) = \mE_1(\mu_t) - \mE_1(\mu_{\tau_t}^\ast)$. First, let us compute the time-derivative $\frac{d}{dt} \mE_1(\mu_{\tau_t}^\ast)$.  
Note that since            
$$
\mE_1(\mu) = -\tau \mH(\mu) + \tau \log Z^+(\mu),
$$
we have for $\tau>0$, 
$$
\begin{aligned}
\frac{d}{d\tau} \mE_1(\mu_{\tau}^\ast) & = -\mH(\mu_\tau^\ast) + \log Z^+(\mu_\tau^\ast)\\
&\qquad + \tau\Big(\langle \log \mu_{\tau}^\ast, \partial_{\tau} \mu_\tau^\ast \rangle + \int_{\mX}\int_{\mY} \mK^+(\mu^\ast_\tau)(y) K(x,y)\frac{d}{d\tau} (\tau^{-1}  \mu^\ast_\tau(x)) dy dx \Big).
\end{aligned}
$$
Moreover, 
$$
\begin{aligned}
   & \tau \int_{\mX}\int_{\mY} \mK^+(\mu^\ast_\tau)(y) K(x,y)\frac{d}{d\tau} (\tau^{-1}  \mu^\ast_\tau(x)) dy dx  = - \int_{\mX}\int_{\mY} \mK^+(\mu^\ast_\tau)(y) \tau^{-1}  K(x,y) dy \mu^\ast_\tau(x) dx
   \\
   & \qquad + \tau\int_{\mX}\int_{\mY} \mK^+(\mu^\ast_\tau)(y) \tau^{-1}  K(x,y) dy \partial_{\tau} \mu^\ast_\tau(dx) \\
   & = - \int_{\mX}\int_{\mY} \mK^+(\mu^\ast_\tau)(y) \tau^{-1}  K(x,y) dy \mu^\ast_\tau(x) dx -\tau \langle \log \mK^-(\mK^+(\mu_{\tau}^\ast)) , \partial_{\tau} \mu^\ast_\tau\rangle 
\end{aligned}
$$
Combining the last two identities leads to 
$$
   \begin{aligned}
\frac{d}{d\tau}\mE_1(\mu_{\tau}^\ast) & = -\mH(\mu_\tau^\ast) + \log Z^+(\mu_\tau^\ast)  
- \int_{\mX}\int_{\mY} \mK^+(\mu^\ast_\tau)(y) \tau^{-1}  K(x,y) dy \mu^\ast_\tau(x) dx\\
& + \underbrace{\tau 
\langle\log \mu_{\tau}^\ast- \log \mK^-(\mK^+(\mu_{\tau}^\ast)), \partial_{\tau} \mu^\ast_\tau\rangle }_{ = \langle \partial_u \mE_1^\tau(u)|_{u= \mu_{\tau}^\ast}, \ \partial_{\tau} \mu^\ast_\tau\rangle = 0}\\
& = \tau^{-1} \mE_1 (\mu^\ast_\tau) - \int_{\mX}\int_{\mY} \mK^+(\mu^\ast_\tau)(y) \tau^{-1}  K(x,y) dy \mu^\ast_\tau(x) dx .
\end{aligned} 
$$
Consequently, we have 
 \begin{equation}\begin{aligned}\label{eq:dh}
\frac{d}{dt}\mE_1(\mu_{\tau_t}^\ast)  & = \tau^{\prime}_t \frac{d}{d\tau}\mE_1(\mu_{\tau}^\ast)\Big|_{\tau  = \tau_t}\\
& =\tau^{\prime}_t \Big( \tau_t^{-1} \mE_1 (\mu^\ast_{\tau_t}) - \int_{\mX}\int_{\mY} \mK^+(\mu^\ast_{\tau_t})(y) \tau_t^{-1}  K(x,y) dy \mu^\ast_{\tau_t}(x) dx \Big).
\end{aligned} \end{equation}
Next, a similar calculation as above applied to $\mE_1 (\mu_t)$ gives 
\begin{equation}\label{eq:dE1tau}
\begin{aligned}
    \frac{d}{dt} \mE_1 (\mu_t) & = \tau^\prime_t \Big(\tau_t^{-1}\mE_1 (\mu_t) - \int_{\mX}\int_{\mY} \mK^+(\mu_t)(y) \tau_t^{-1}  K(x,y) dy \mu_t(x) dx\Big)\\
    & \qquad +  \tau_t \int_{\mX} \log \Big(\frac{d\mu_t}{d \mK^-(\mK^+(\mu_t))}\Big) \cdot \partial_t \mu_t (dx)\\
    & = \tau^\prime_t \Big(\tau_t^{-1}\mE_1 (\mu_t)  -\int_{\mX}\int_{\mY} \mK^+(\mu_t)(y) \tau_t^{-1}  K(x,y) dy \mu_t(x) dx\Big)\\
    & \qquad -  \tau_t^2 \int_{\mX} \nabla_x \log \Big(\frac{d\mu_t}{d \mK^-(\mK^+(\mu_t))}\Big) \cdot \nabla_x \log \Big(\frac{d\mu_t}{d \mK^-(\nu_t)}\Big) \mu_t(dx).
\end{aligned}    
\end{equation}
As a result of \eqref{eq:dE1tau} and \eqref{eq:dh}, 
\begin{equation}\label{eq:dE1tau2}
\begin{aligned}
 & \frac{d}{dt}\mL_1(\mu_t)  =   \frac{d}{dt} (\mE_1(\mu_t) - \mE_1 (\mu^\ast_{\tau_t}))\\
 & = \frac{\tau_t^\prime}{\tau_t} \Big[\mL_1(\mu_t) - \Big(\int_{\mX}\int_{\mY} \mK^+(\mu_t)(y)  K(x,y) dy \mu_t(x) dx - \int_{\mX}\int_{\mY} \mK^+(\mu^\ast_{\tau_t})(y) K(x,y) dy \mu^\ast_{\tau_t}(x) dx\Big) \Big]\\
    & \qquad -  \tau_t^2 \int_{\mX} \nabla_x \log \Big(\frac{d\mu_t}{d \mK^-(\mK^+(\mu_t))}\Big) \cdot \nabla_x \log \Big(\frac{d\mu_t}{d \mK^-(\nu_t)}\Big) \mu_t(dx)\\
    & \leq \frac{\tau_t^\prime}{\tau_t} \Big(\mL_1(\mu_t) + 2\|K\|_{\infty} \Big) -\frac{\tau_t^2}{2} \mI(\mu_t | \mK^-(\mK^+(\mu_t))) + K_{xy}^2  \diam (\mY)^2\cdot \mH(\nu_t | \mK^+(\mu_t)).
\end{aligned}    
\end{equation}
Next, we have from the ascent dynamics for $\nu_t$,
\begin{equation}\label{eq:dL2tau}
\begin{aligned}
 \frac{d}{dt} \mL_2 (\mu_t,\nu_t) &= \tau^\prime_t  \mH(\nu_t | \mK^+(\mu_t)) + \tau_t \int_{\mY} \log\Big(\frac{d\nu_t}{d \mK^+(\mu_t)}  \Big)\partial_t \nu_t(dy) - \tau_t \int_{\mY} \frac{d}{dt} (\log \mK^+ (\mu_t) )\nu_t(dy)\\
 & \leq \tau^\prime_t  \mH(\nu_t | \mK^+(\mu_t)) -\eta_t\tau_t^2 \mI(\nu_t| \mK^+(\mu_t))- \tau_t \int_{\mY} \frac{d}{dt} (\log \mK^+ (\mu_t) )\nu_t(dy).
\end{aligned}    
\end{equation}
The third term on the RHS above is 
$$
\begin{aligned}
& -\tau_t \Big(\int_{\mY} \int_{\mX} \frac{d}{dt} (\tau_t^{-1} \mu_t(x)) K(x,y)dx \nu_t(y) dy - \frac{d}{dt}\log (Z^+(\mu_t)) \Big)\\
& = -\tau_t \Big(\int_{\mY} \int_{\mX} \frac{d}{dt} (\tau_t^{-1} \mu_t(x)) K(x,y)dx (\nu_t(y) - \mK^+(\mu_t)(y)) dy \Big)\\
& = - \frac{\tau^\prime_t }{\tau_t}\int_{\mY} \int_{\mX}  \mu_t(x) K(x,y)dx (\nu_t(y) - \mK^+(\mu_t)(y)) dy\\
 & \qquad \qquad - \int_{\mY} \int_{\mX} \partial_t \mu_t(x) K(x,y)dx (\nu_t(y) - \mK^+(\mu_t)(y)) dy\\
 & = - \frac{\tau^\prime_t }{\tau_t}\int_{\mY} \int_{\mX}  \mu_t(x) K(x,y)dx (\nu_t(y) - \mK^+(\mu_t)(y)) dy\\
 & \qquad + \tau_t^2 \int_{\mX} \nabla_x \log \Big(\frac{d\mK^-(\mK^+(\mu_t))}{d\mK^-(\nu_t)}\Big) \cdot \nabla_x \log \Big(\frac{d\mu_t}{d\mK^-(\nu_t)}\Big)d\mu_t(x)
 \\
& \leq \frac{2\|K\|_{\infty} |\tau^{\prime}_t|}{\tau_t} + \frac{\tau_t^2}{2} \mI(\mu_t |\mK^-(\mK^+(\mu_t)) ) + 3 K_{xy}^2 \diam(\mY)^2\cdot \mH(\nu_t | \mK^+(\mu_t)).
\end{aligned}
$$
Combining the last two displays yields
\begin{equation}\label{eq:dmL2}
\begin{aligned}
    &  \frac{d}{dt} \mL_2 (\mu_t,\nu_t) \leq  \Big(\tau^\prime_t + 3 K_{xy}^2 \diam(\mY)^2\Big) \mH(\nu_t | \mK^+(\mu_t)) -\eta_t\tau_t^2 \mI(\nu_t| \mK^+(\mu_t)) \\
     & \qquad + \frac{2\|K\|_{\infty} |\tau^{\prime}_t|}{\tau_t} + \frac{\tau_t^2}{2} \mI(\mu_t |\mK^-(\mK^+(\mu_t)) ).
\end{aligned}
\end{equation}
It then follows from \eqref{eq:dE1tau2} \eqref{eq:dmL2} that 
$$
\begin{aligned}
   & \frac{d}{dt} \mL( \mu_t,\nu_t) \leq -\frac{\tau_t^2(1-\gamma)\lambda_{LS}(\tau_t)}{2} \mH(\mu_t |\mK^-(\mK^+(\mu_t)) ) + \Big(\gamma\tau^{\prime}_t + (1+3\gamma) K_{xy}^2 \diam(\mY)^2\Big)\mH(\nu_t | \mK^+(\mu_t)) \\
    & + \frac{2(1+\gamma)\|K\|_{\infty} |\tau^{\prime}_t|}{\tau_t} + \frac{\tau^{\prime}_t}{\tau_t}\mL_1(\mu_t)-\gamma \eta_t\tau_t^2\mI(\nu_t| \mK^+(\mu_t))\\
    &\leq -\Big(\frac{(1-\gamma)\tau_t\lambda_{LS}(\tau_t)}{2} - \frac{\tau^{\prime}_t}{\tau_t}\Big) \mL_1(\mu_t) -\Big(\gamma \eta_t\tau_t  \lambda_{LS}(\tau_t)-\frac{\gamma\tau^{\prime}_t}{\tau_t} - \frac{(1+3\gamma) K_{xy}^2 \diam(\mY)^2}{\tau_t}\Big)  \mL_2(\mu_t,\nu_t)\\
    & + \frac{2(1+\gamma)\|K\|_{\infty} |\tau^{\prime}_t|}{\tau_t}\\
    & \leq -\alpha_t \mL(\mu_t,\nu_t) + \frac{2(1+\gamma)\|K\|_{\infty}}{t \log t },
\end{aligned}
$$
where the time-dependent rate satisfies for large $t$ that
$$\begin{aligned}
\alpha_t & = \Big(\Big(\frac{(1-\gamma)\tau_t\lambda_{LS}(\tau_t)}{2}-\frac{\tau^{\prime}_t}{\tau_t}\Big) \wedge \Big(\eta_t \tau_t \lambda_{LS}(\tau_t)-\frac{\tau^{\prime}_t}{\tau_t} - \frac{(1+3\gamma) K_{xy}^2 \diam(\mY)^2}{\gamma \tau_t}\Big)\Big)\\
& \geq \frac{\xi}{\log t}\Big(C_1 t^{-\xi_{\ast}/\xi} \wedge \Big(\eta_t t^{-\xi^\ast/\xi}+\frac{C_2}{t \log t}  - \frac{C_3}{(\log t)^2}\Big)\Big).
\end{aligned}$$
Note that in the above we have used the decreasing of $\tau_t$ for large $t$. 
By choosing $\eta_t \geq M \frac{t^{\xi^\ast/\xi}}{(\log t)^2}$ for some large $M>0$, we have 
$$
\alpha_t \geq \frac{C t^{-\xi^\ast/\xi-\epsilon}}{\log t}. 
$$
As a result, we have obtained that for any $\epsilon^\prime <1-\xi^\ast/\xi$ and for  $t$ large enough, 
$$
\frac{d}{dt} \mL( \mu_t,\nu_t) \leq -C t^{-\xi^\ast/\xi-\epsilon^\prime}\mL( \mu_t,\nu_t) + \frac{C^\prime}{t \log t}. 
$$
Define $Q(t) =  \mL( \mu_t,\nu_t)  - \frac{C^\prime}{C} t^{-1 + \xi^\ast/\xi+\epsilon^\prime}$. Then it is straightforward to show that for $t > t_\ast$ large enough,
$$
\frac{d}{dt} Q(t) \leq -C t^{-\xi^\ast/\xi-\epsilon^\prime} Q(t),
$$
which  implies that 
\begin{equation}\label{eq:bdLmunu}
    \mL(\mu_t,\nu_t) \leq Q(t_\ast) e^{-\frac{C}{1-\xi^\ast/\xi-\epsilon} (t^{1-\xi^\ast/\xi-\epsilon^\prime} - t_\ast^{1-\xi^\ast/\xi-\epsilon^\prime})} + \frac{C^\prime}{C} t^{-1 + \xi^\ast/\xi+\epsilon^\prime} \leq C^{\prime\prime} t^{-1 + \xi^\ast/\xi+\epsilon^\prime}
\end{equation}
since $\xi^\ast < \xi$ and $Q(t_\ast)$ is finite. By the definition of $\mL$ and the fact that $\log t \ll t^\epsilon$ for any $\epsilon>0$, the last estimate further implies that for any $0<\epsilon<1-\xi^\ast/\xi$,
\begin{equation}\label{eq:bdHmutstar}
  \mH(\mu_t | \mu_{\tau_t}^\ast) \leq C^{\prime\prime} t^{-1 + \xi^\ast/\xi+\epsilon}, \text{ for } t> t_\ast.  
\end{equation}
In addition, using the same arguments used in the proof of the second bound in \eqref{eq:tauHnut} from Theorem \ref{thm:main2}, one can obtain that 
$$
\mH(\nu_t | \nu_{\tau_t}^\ast) \leq C^{\prime\prime} t^{-1 + \xi^\ast/\xi+\epsilon}, \text{ for } t> t_\ast. 
$$

\textbf{Step 2: }  Bounding $\NI(\mu_t, \nu_t)$. Let us first claim that the difference $\NI(\mu_t, \nu_t) - \NI(\mu_{\tau_t}^\ast, \nu_{\tau_t}^\ast)$ satisfies 
\begin{equation}\label{eq:diffNI}
     \NI(\mu_t, \nu_t) - \NI(\mu_{\tau_t}^\ast, \nu_{\tau_t}^\ast) \leq Ct^{-\frac{1-\xi^\ast/\xi -\epsilon}{2}}. 
\end{equation}

In fact, by definition, 
$$
\begin{aligned}
   &  \NI(\mu_t, \nu_t) - \NI(\mu_{\tau_t}^\ast, \nu_{\tau_t}^\ast) = \max_{\nu\in \mP(\mX)}E_0(\mu_t, \nu) -\max_{\nu\in \mP(\mX)}E_0(\mu_{\tau_t}^\ast, \nu) + \min_{\mu \in \mP(\mX)} E_0(\mu, \nu_{\tau_t}^\ast)  - \min_{\mu \in \mP(\mX)} E_0(\mu, \nu_t)\\
    & = \max_{y\in \mY}\int_{\mX} K(x,y)\mu_t(dx) -\max_{y\in \mY}\int_{\mX} K(x,y)\mu_{\tau_t}^\ast(dx)\\
    & \qquad + \min_{x \in \mX} \int_{\mY} K(x,y) \nu_{\tau_t}^\ast(dy)  - \min_{x\in \mX} \int_{\mY} K(x,y) \nu_t(dy)\\
    & =: J_1 + J_2. 
\end{aligned} 
$$
To bound $J_1$, let us define $ y_t \in  \argmax_{y\in \mY}\int_{\mX} K(x,y)\mu_t(dx) $ and $y^\ast_t \in \argmax_{y\in \mY}\int_{\mX} K(x,y)\mu_{\tau_t}^\ast (dx) $. Then by the optimality of $y^\ast_t$ and the boundedness of $K$, 
$$
\begin{aligned}
J_1 & = \int_{\mX} K(x,y_t )\mu_t(dx) - \int_{\mX} K(x,y^\ast_t)\mu_{\tau_t}^\ast(dx)\\
& \leq \int_{\mX} K(x,y_t )\mu_t(dx) - \int_{\mX} K(x,y_t)\mu_{\tau_t}^\ast(dx)\\
& \leq \|K\|_{\infty} \text{TV} (\mu_t, \mu_{\tau_t}^\ast)\\
& \leq \sqrt{2}\|K\|_{\infty} \sqrt{\mH(\mu_t| \mu_{\tau_t}^\ast)}\\
& \leq C t^{-\frac{1-\xi^\ast/\xi-\epsilon}{2}} . 
\end{aligned} 
$$
The same bound holds for $J_2$ after applying a similar argument as above, which completes the proof of \eqref{eq:diffNI}.
Finally,  applying Lemma \ref{lem:laplace} with $\tau = \tau_t = \xi/\log t$, we have for $t$  large enough, 
\begin{equation}\label{eq:NIstar}
  \NI(\mu_{\tau_t}^\ast, \nu_{\tau_t}^\ast) \leq \frac{C \log (\log t)}{\log t }.  
\end{equation}
Hence the estimate \eqref{eq:NImut} follows from \eqref{eq:diffNI} and \eqref{eq:NIstar}.

\end{proof}

We restate \cite[Theorem 5]{domingo2020mean} in the following lemma which is used to control $ \NI(\mu_{\tau}^\ast, \nu_{\tau}^\ast)$. 
\begin{lemma}\cite[Theorem 5]{domingo2020mean}\label{lem:laplace}
Let $\epsilon >0, \delta := \epsilon/(2\text{Lip}(K))$ and let $V_{\delta}$ be a lower bound on the volume of a ball of radius of $\delta$ in $\mX$ and $\mY$. Let $(\mu_{\tau}^\ast, \nu_{\tau}^\ast)$ be the solution of the solution of \eqref{eq:fixedpoint}. Then $(\mu_{\tau}^\ast, \nu_{\tau}^\ast)$ is an $\epsilon$-Nash equilibrium of $E_0$ if  
$$
\tau \leq \frac{\epsilon}{4 \log\Big(\frac{2(1-V_{\delta})}{V_{\delta}} (4\|K\|_{\infty}/\epsilon-1)\Big)}.
$$
In particular, when $\mX$ and $\mY$ are Riemannian manifolds or Euclidean tori of dimensions $d_{\mX}$ and $d_{\mY}$ respectively so that $\text{Vol}(B_{x}^\delta) \geq C \delta^{d_{\mX}}$ and $\text{Vol}(B_{y}^\delta) \geq C \delta^{d_{\mY}}$, the bounds above become
$$
\tau \leq \frac{C\epsilon}{\log \epsilon^{-1}}
$$
for some constant $C>0$ depending only on $K,d_{\mX}, d_{\mY}$. Alternatively, if $\tau$ is sufficiently small, then $(\mu_{\tau}^\ast, \nu_{\tau}^\ast)$ is an $\epsilon$-Nash equilibrium with 
\begin{equation}\label{eq:epsNash}
  \epsilon = \beta \tau(\log (1/\tau)) \text{ for } \beta > d_{\mX}\vee d_{\mY} +1.
\end{equation}
\end{lemma}

\section{Conclusion and future work} \label{sec:conclusion}
In this paper, we proved the global exponential convergence of two-scale Mean-Field GDA dynamics for finding the unique MNE of the entropy-regularized game objective on the space of probability measures. We also proved that the convergence of the annealed GDA dynamics to the MNE of the unregularized game objective with respect to the Nikaid\` o-Isoda error. The key ingredient of our proofs are new Lyapunov functions which are used to capture the dissipation of the two-scale GDA dynamics in different scaling regimes. 

We would like to mention several open problems and future research directions. First, although we have proved the global convergence of the Mean-Field GDA in  the fast ascent/descent regime (with a finite but large/small time-scale ratio), it remains an open question to show the convergence or nonconvergence of the Mean-Field GDA in the intermediate time-scale  regime, including particularly the case where no time-scale separation occurs (i.e. $\eta =1$ in \eqref{eq:WGDA}). Also, currently our  results only hold on bounded state spaces and the convergence rate depends on the uniform bounds of the potential function $K$ on the state spaces. It remains an important question to extend the results to minmax optimization on unbounded state spaces.  In practice, the Mean-Field GDA needs to be implemented by a certain interacting particle algorithm such as \eqref{eq:gdaparticle} with a large number of particles. Existing result on the mean field limit of \eqref{eq:gdaparticle} (e.g. \cite[Theorem 3]{domingo2020mean}) only holds in finite time interval and the error bound can potentially grow exponentially in time. To obtain a quantitative error analysis of the GDA for minmax optimization, it is an interesting open question to derive a uniform-in-time quantitative error bound between the particle system and the mean field dynamics.  Finally, we anticipate our results can be exploited to prove theoretical convergence guarantees for a variety of minmax learning problems, including especially the training of GANs \cite{goodfellow2020generative}, adversarial learning problems \cite{madry2018towards}, dual training of energy based models \cite{dai2019exponential,domingo2021dual} and weak adversarial networks for PDEs \cite{zang2020weak}.

\section*{Acknowledgments}
The author thanks Lexing Ying for suggesting the question  and thanks Jianfeng Lu, Yiping Lu and  Chao Ma for helpful discussions at the early stage of the problem setup. The author also thank  Fei Cao and L\'ena\"ic Chizat for the valuable feedback on an early version of the paper. This work is supported by the National Science Foundation through the award DMS-2107934. 

\appendix
\section*{Appendix}

\section{Proofs of preliminary lemmas}\label{sec:applem}

\begin{proof}[Proof of Lemma \ref{lem:bdL}]
First we have from the definition of $E$ and \eqref{eq:mK2} that 
$$\begin{aligned}
\mL_2 (\mu,\nu) & =
\max_{\nu\in \mP(\mY)}  E(\mu, \nu) - E(\mu, \nu)\\
& = \tau\log Z^+(\mu) -  \int_{\mX}\int_{\mY} K(x,y)\mu(dx) \nu(dy) - \tau \mH(\nu)\\
& = \tau \mH(\nu | \mK^+ (\mu))
\end{aligned}
$$
which proves \eqref{eq:ml2}. 
To see \eqref{eq:ml4}, notice again from \eqref{eq:mK2} that 
$$\begin{aligned}
\mL_4 (\mu,\nu) & =
E(\mu,\nu)  - \min_{\mu\in \mP(\mX)} E(\mu,\nu) \\
& = \int_{\mX}\int_{\mY} K(x,y)\mu(dx) \nu(dy) - \tau \mH(\mu) + \tau\log Z^-(\nu)  \\
& = \tau\mH(\mu | \mK^-(\nu)).
\end{aligned}
$$
Next, we prove \eqref{eq:ml1}. In fact, replacing $\nu$ by $\mK^-(\mu)$ in the last display leads to 
$$
\begin{aligned}
\tau\mH(\mu | \mK^-(\mK^+ (\mu))) & = \mL_4(\mu, \mK^+ (\mu))\\
& = E(\mu, \mK^+(\mu)) - \min_{\mu\in \mP(\mX)} E(\mu, \mK^+(\mu))\\
& = \max_{\nu\in \mP(\mY)}E(\mu, \nu) - \min_{\mu\in \mP(\mX)} E(\mu, \mK^+(\mu))\\
& \geq \max_{\nu\in \mP(\mY)}E(\mu, \nu) - \min_{\mu\in \mP(\mX)}\max_{\nu\in \mP(\mY)} E(\mu, \nu)\\
& = \mL_1(\mu),
\end{aligned}
$$
which proves the lower bound part of \eqref{eq:ml1}. To prove the upper bound part, we first compute the first-variation of the log partition function $\log Z^+(\mu)$ as 
\begin{equation}\label{eq:dlogZp}
    \frac{\partial}{\partial \mu} \log Z^+(\mu) = \tau^{-1} \int_{\mY} K(x,y) d \mK^{+}(\mu)(y) = - \log (\mK^-(\mK^+(\mu))) - \log Z^-(\mK^+(\mu)).
\end{equation}
Then an application of the convexity of $\mu \gt \log Z^+(\mu)$ shown in Lemma \ref{lem:convpart} yields $$
\begin{aligned}
\log Z^+(\mu) - \log Z^+(\mu^\ast) & \geq \int_{\mX} \frac{\partial \log Z^+(\mu)}{\partial \mu}\Big|_{\mu = \mu^\ast} (d\mu -d\mu^\ast)\\
& = -\int_{\mX}\log (\mK^-(\mK^+(\mu^\ast))) (d\mu -d\mu^\ast)\\
& = -\int_{\mX}\log (\mu^\ast) (d\mu -d\mu^\ast).
\end{aligned}
$$
Consequently, one obtains that 
$$
\begin{aligned}
\mL_1(\mu) & = \mE_1(\mu)  - \mE_1(\mu^\ast)\\
&= -\tau( \mH(\mu) - \mH(\mu^\ast)) + \tau (\log Z^+(\mu) - \log Z^+(\mu^\ast))\\
& \geq \tau \int \log \mu d\mu - \int \log \mu^\ast d\mu^\ast -\int_{\mX}\log (\mu^\ast) (d\mu -d\mu^\ast)\\
& = \tau \mH(\mu | \mu^\ast).
\end{aligned}$$

Finally, the inequality \eqref{eq:ml3} follows from the same reasoning above by exploiting the convexity of $\nu \gt \log Z^-(\nu)$.

\end{proof}

The following lemma establishes the convexity of the log partition functions $\log Z^+(\mu)$ and  $\log Z^-(\nu)$. It was proved in \cite[Proposition 3]{ma2021provably}, but for completeness we include its proof here.
\begin{lemma}\label{lem:convpart}
Both the functional $\mu \gt \log Z^+(\mu)$ and the functional $\nu \gt \log Z^-(\nu)$ are convex.
\end{lemma}

\begin{proof}
We only present the proof for the convexity of the map $\mu \gt \log Z^+(\mu)$ since the other case can be proved in the same manner. For any $\mu_1,\mu_2\in \mP(\mX)$ and $\alpha \in (0,1)$,  
$$
\begin{aligned}
    &Z^+(\alpha \mu_1 + (1-\alpha)\mu_2)\\
    & = \log \Big(\int_{\mY} \exp\Big(\int_{\mX} \tau^{-1} K(x,y) (\alpha \mu_1(dx)+ (1-\alpha) \mu_2(dx))\Big)dy\Big)\\
    & = \log \Big(\int_{\mY} \exp\Big(\alpha \int_{\mX} \tau^{-1} K(x,y)  \mu_1(dx)\Big) \cdot \exp\Big( (1-\alpha) \int_{\mX} \tau^{-1} K(x,y) \mu_2(dx)\Big)dy\Big)\\
    & \leq \log \Big(\Big(\int_{\mY} \exp\Big( \int_{\mX} \tau^{-1} K(x,y)  \mu_1(dx)\Big) dy \Big)^\alpha \cdot \Big(\int_{\mY} \exp\Big(  \int_{\mX} \tau^{-1} K(x,y) \mu_2(dx)\Big)dy\Big)^{1-\alpha}\Big)\\
    & = \alpha Z^+( \mu_1) +  (1-\alpha)Z^+( \mu_2). 
\end{aligned}
$$

\end{proof}

\begin{proof}[Proof of Proposition \ref{prop:diffL}]
First let us compute the functional gradient 
$
\frac{\partial \mL_1}{\partial u}(\mu_t). 
$
In fact, thanks to the fact that $\mL_1(\mu) = \mE_1(\mu) - \mE_1(\mu^\ast)$ and \eqref{eq:mK3}, one has that 
$$
\begin{aligned}
\frac{\partial \mL_1}{\partial u} \Big|_{\mu = \mu_t} & = -\tau\frac{\partial }{\partial u} \Big(\mH(\mu) - \log Z^+(\mu)\Big)  \Big|_{\mu = \mu_t}\\
& = \tau\Big( \log \mu + \tau^{-1} \int_{\mY} K(x,y) d \mK^{+}(\mu)(y) \Big)\Big|_{\mu = \mu_t}\\
& = \tau\Big(\log \Big(\frac{d\mu_t}{d \mK^-(\mK^+(\mu_t))} \Big) - \log Z^-(\mK^+(\mu_t))\Big), 
\end{aligned}
$$
where in the second identity we have used \eqref{eq:dlogZp}. 
Therefore it follows from \eqref{eq:WGDA} and the Cauchy-Schwarz inequality that 
$$
\begin{aligned}
 \frac{d}{dt} \mL_1(\mu_t)& = \int_{\mX} \frac{\partial \mL_1}{\partial u}(\mu_t) d (\partial_t \mu_t)\\
 & = \tau^2\int_{\mX} \Big(\log \Big(\frac{d\mu_t}{d \mK^-(\mK^+(\mu_t))} \Big) - \log Z^-(\mK^-(\mK^+(\mu_t)))\Big)  \nabla_x \cdot \Big(\mu_t\nabla_x \log\Big(\frac{d\mu_t}{d \mK^-(\nu_t)}\Big)\Big)dx\\
 & = -\tau^2 \int_{\mX} \nabla_x \log \Big(\frac{d\mu_t}{d \mK^-(\mK^+(\mu_t))}\Big) \cdot \nabla_x \log\Big(\frac{d\mu_t}{d \mK^-(\nu_t)}\Big) d\mu_t(x)\\
 & = -\tau^2\int_{\mX} \Big|\nabla_x \log \Big(\frac{d\mu_t}{d \mK^-(\mK^+(\mu_t))}\Big) \Big|^2d\mu_t(x) \\
 & \qquad - \tau^2\int_{\mX}  \nabla_x \log \Big(\frac{d \mK^-(\mK^+(\mu_t))}{d  \mK^-(\nu_t)}\Big)  \cdot \nabla_x \log \Big(\frac{d\mu_t}{d \mK^-(\mK^+(\mu_t))}\Big) d\mu_t(x)\\
 & \leq -\frac{\tau^2}{2} \mI(\mu_t | \mK^-(\mK^+(\mu_t))) + \frac{\tau^2}{2} \int_{\mX}  \Big|\nabla_x \log \Big(\frac{d \mK^-(\mK^+(\mu_t))}{d  \mK^-(\nu_t)}\Big)\Big|^2 d\mu_t(x).
\end{aligned}
$$
Furthermore, using the fact that 
$$
\mK^-(\mK^+(\mu_t)) \propto \exp\Big(-\int_{\mY} \tau^{-1} K(x,y) d\mK^+(\mu_t) (y) \Big),\quad \mK^-(\nu_t) \propto \exp\Big(-\int_{\mY} \tau^{-1} K(x,y) d\nu_t(y) \Big),
$$
one derives that 
\begin{equation}
    \begin{aligned}\label{eq:reskl}
    \frac{1}{2} \int_{\mX}  \Big|\nabla_x \log \Big(\frac{d \mK^-(\mK^+(\mu_t))}{d  \mK^-(\nu_t)}\Big)\Big|^2 d\mu_t(x) & = \frac{1}{2\tau^2} \int_{\mX}  \Big|\int_{\mY} \nabla_x K(x,y) (d\mK^+(\mu_t)(y) - d\nu_t(y)) \Big|^2 d\mu_t(x).
\end{aligned}
\end{equation}
Moreover, using the mean value theorem, one has for any $y_0\in \mY$ that 
$$
\nabla_x K(x,y)-\nabla_x K(x,y_0) = \Big[\int_0^1 \nabla_{xy}^2 K(x,y_0 + s(y-y_0)) s\, ds \Big](y - y_0 ).
$$
Inserting the last identity into \eqref{eq:reskl} leads to 
\begin{equation}
    \begin{aligned}\label{eq:reskl2}
    \frac{1}{2} \int_{\mX}  \Big|\nabla_x \log \Big(\frac{d \mK^-(\mK^+(\mu_t))}{d  \mK^-(\nu_t)}\Big)\Big|^2 d\mu_t(x) 
    & \leq \frac{K_{xy}^2  \diam (\mY)^2}{2\tau^2}  \text{TV}^2(\mK^+(\mu_t), \nu_t)\\
    & \leq  \frac{K_{xy}^2  \diam (\mY)^2}{\tau^2}  \mH(\nu_t| \mK^+(\mu_t)),
\end{aligned}
\end{equation}
where we have used the Pinsker's inequality in the last inequality above. Combining the last two estimates proves \eqref{eq:dtL1}.

Next we proceed with the proof of \eqref{eq:dtL2}. Recall from \eqref{eq:ml2} that 
$$
\mL_2(\mu,\nu) = \tau \mH(\nu | \mK^+ (\mu)) = \tau \int_{\mY} \log \Big( \frac{\nu}{\mK^+(\mu)}\Big) d\nu.
$$
Hence 
\begin{equation}\label{eq:dL2-1}
  \begin{aligned}
\frac{d}{dt}\mL_2(\mu_t,\nu_t) = \tau \Big(\int_{\mY} \log \Big( \frac{\nu_t}{\mK^+(\mu_t)}\Big) d(\partial_t \nu_t) - \int_{\mY} \frac{d}{dt} \log (\mK^+(\mu_t)) d\nu_t \Big)
\end{aligned}  
\end{equation}
Using the second equation of \eqref{eq:WGDA} that $\nu_t$ solves, one sees that the first term on the right side above becomes
\begin{equation}\label{eq:dL2-2}
    \tau \Big(\int_{\mY} \log \Big( \frac{\nu_t}{\mK^+(\mu_t)}\Big) d(\partial_t \nu_t) = -\eta\tau^2 \mI(\nu_t | \mK^+(\mu_t)). 
\end{equation}
To compute the second term on the right side  of \eqref{eq:dL2-1}, observe that 
$$
\log \mK^+(\mu_t) = \tau^{-1}\int_{\mX} K(x,y) d\mu_t(x) - \log (Z^+(\mu_t)). 
$$
An a result of above and \eqref{eq:dlogZp}, one obtains that 
$$
\begin{aligned}
     & -\tau \int_{\mY} \frac{d}{dt} \log (\mK^+(\mu_t)) d\nu_t  \\
     & = -\tau \int_{\mY} \int_{\mX} \Big(\tau^{-1} K(x,y)  - \frac{\partial \log (Z^+(\mu))}{\partial \mu} (\mu_t)\Big) d (\partial_t \mu_t(x)) d\nu_t(y)\\
    & = -\tau  \int_{\mX}\Big(\int_{\mY} \tau^{-1} K(x,y) d\nu_t(y) - \int_{\mY}  \tau^{-1} K(x,y) d\mK^+(\mu_t)(y) \Big) d (\partial_t \mu_t(x)) \\
    & =-\tau  \int_{\mX}\Big(\log (\mZ^-(\mK^+(\mu_t)) \mK^-(\mK^+(\mu_t)) ) - \log (\mZ^-(\nu_t) \mK^-(\nu_t)) \Big) d (\partial_t \mu_t(x)) \\
    & = \tau^2\int_{\mX} \nabla_x \log 
    \Big(\frac{d \mK^-(\mK^+(\mu_t))}{d\mK^-(\nu_t)} \Big) \cdot \nabla_x \log 
    \Big(\frac{d\mu_t}{ d \mK^-(\nu_t)}\Big) d\mu_t(x)\\
    & = \tau^2\Big(\int_{\mX} \Big|\nabla_x \log 
    \Big(\frac{d \mK^-(\mK^+(\mu_t))}{d\mK^-(\nu_t)} \Big) \Big|^2 d\mu_t(x)\\
    & \qquad + \int_{\mX} \nabla_x \log 
    \Big(\frac{d \mK^-(\mK^+(\mu_t))}{d\mK^-(\nu_t)} \Big) \cdot \nabla_x \log 
    \Big(\frac{d \mu_t}{ d \mK^-(\mK^+(\mu_t))}\Big) d\mu_t(x)\Big)\\
    & \leq \frac{\tau^2}{2} \mI(\mu_t | \mK^-(\mK^+(\mu_t))) + \frac{3\tau^2}{2} \int_{\mX} \Big|\nabla_x \log 
    \Big(\frac{d \mK^-(\mK^+(\mu_t))}{d\mK^-(\nu_t)} \Big) \Big|^2 d\mu_t(x)\\
    & \leq \frac{\tau^2}{2} \mI(\mu_t | \mK^-(\mK^+(\mu_t))) + 3K_{xy}^2\diam (\mY)^2 \mH(\nu_t| \mK^+(\mu_t)),
\end{aligned}
$$
where we have used \eqref{eq:reskl2} in the last inequality above. The estimate \eqref{eq:dtL2} then follows from above, \eqref{eq:dL2-2} and \eqref{eq:dL2-1}. This completes the proof of the proposition. 
\end{proof}

\begin{proof}[Proof of Lemma \ref{lem:laplace}]
The proof of the lemma can be found in \cite[Appendix C.4]{domingo2020mean}. We would like to elaborate on the proof of \eqref{eq:epsNash}. In fact, by tracking the proof of \cite[Theorem 5]{domingo2020mean}, one sees that $(\mu_{\tau}^\ast, \nu_{\tau}^\ast)$ is an $\epsilon$-Nash equilibrium provided that 
$$
e^{\frac{\eps}{2\tau}} \Big( \frac{\text{Vol}(B_{x}^\delta)}{1- \text{Vol}(B_{x}^\delta)}  \vee  \frac{\text{Vol}(B_{y}^\delta)}{1- \text{Vol}(B_{y}^\delta)}\Big) \geq 2 (4\|K\|_{\infty}/\epsilon - 1),
$$
where we recall that $\delta = \epsilon/(2\text{Lip}(K))$. Thanks to the lower bound on the volume of small balls in $\mX$ and $\mY$, the inequality above holds if 
$$
e^{\frac{\epsilon}{2\tau}} \geq C_1 \epsilon^{-(d_{\mX}\vee d_{\mY} +1)}
$$
for some $C_1>0$ depending only on $K,d_{\mX}, d_{\mY}$.  Now with the choice \eqref{eq:epsNash} for $\epsilon$ with $\beta > d_{\mX}\vee d_{\mY} +1$, one has for $\tau$ sufficiently small that 
$$
\begin{aligned}
    e^{\frac{\eps}{2\tau}} \geq  \frac{1}{\tau^{\beta}} > \frac{C_1}{\beta^{d_{\mX} \vee d_{\mY} + 1} }  \Big(\frac{1}{\tau \log(1/\tau)}\Big)^{d_{\mX} \vee d_{\mY} + 1} = C_1 \epsilon^{-(d_{\mX}\vee d_{\mY} +1)}.
\end{aligned}
$$
\end{proof}

\section{Proofs of convergence for Mean-Field GDA in the fast descent regime} \label{sec:appmfgda}
Let us first state a proposition which bounds the time-derivative of $\mL_2(\nu_t)$ and $\mL_4(\mu_t,\nu_t)$. 

\begin{proposition}\label{prop:ml24} Let $(\mu_t,\nu_t)$ be the solution to the DGA dynamics \eqref{eq:WGDA}. Then 
    \begin{equation}\label{eq:dmL3}
        \frac{d}{dt} \mL_3(\nu_t) \leq -\frac{\eta\tau^2}{2} \mI(\nu_t | \mK^+(\mK^-(\nu_t))) + \eta K_{xy}^2 \diam(\mX)^2  \cdot \mH(\mu_t | \mK^-(\nu_t))
    \end{equation}
    and 
    \begin{equation}\label{eq:dmL4}
     \frac{d}{dt} \mL_4(\mu_t, \nu_t) \leq -\tau^2 \mI(\mu_t | \mK^-(\nu_t)) + \frac{\eta \tau^2}{2} \mI(\nu_t | \mK^+(\mK^-(\nu_t))) + 3\eta K_{xy}^2 \diam(\mX)^2\cdot \mH(\mu_t | \mK^-(\nu_t)).
      \end{equation}
\end{proposition}

\begin{proof}
The proof follows essentially the same reasoning as the proof of Proposition \ref{prop:diffL}. In fact, it can be shown by a straightforward calculation that 
$$
\begin{aligned}
     \frac{d}{dt} \mL_2(\nu_t) & = -\eta \tau^2 \int_{\mY} \nabla_{y} \log \Big(\frac{d\nu_t}{d \mK^+(\mK^-(\nu_t))}\Big) \cdot \nabla_{y} \log \Big(\frac{d\nu_t}{d \mK^+(\mu_t)}\Big)d\nu_t(y)\\
     & \leq -\frac{\eta \tau^2}{2} \mI(\nu_t | \mK^+(\mK^-(\nu_t))) +
     \frac{\eta\tau^2}{2} \int_{\mY} \Big|\nabla_y \log \Big(\frac{d \mK^+(\mK^-(\nu_t))}{d \mK^+(\mu_t)}\Big)\Big|^2 d\nu_t(y)
\end{aligned}
$$
Similar to \eqref{eq:reskl2}, one can obtain that 
$$
\begin{aligned}
\frac{\eta\tau^2}{2} \int_{\mY} \Big|\nabla_y \log \Big(\frac{d \mK^+(\mK^-(\nu_t))}{d \mK^+(\mu_t)}\Big)\Big|^2 d\nu_t(y) 
\leq \eta K_{xy}^2 \diam(\mX)^2 \cdot \mH(\mu_t | \mK^-(\nu_t)).
\end{aligned}
$$
Combining the last two inequalities yield \eqref{eq:dmL3}.

As for time-derivative of $\mL_4(\mu_t,\nu_t)$, one has that
$$
\begin{aligned}
\frac{d}{dt} \mL_4(\mu_t,\nu_t)  =  - \tau^2 \mI(\mu_t | \mK^-(\nu_t)) 
-\eta\tau^2 \int_{\mY} \nabla_y \log \Big(\frac{d \mK^+(\mu_t)}{d \mK^+(\mK^-(\nu_t))}\Big)
\cdot \nabla_y \log \Big(\frac{d\nu_t}{d \mK^+(\mu_t)}\Big)d\nu_t(y). 
\end{aligned}
$$
The second term on the right side above is bounded by 
$$
\begin{aligned}
   &  \eta\tau^2 \mI(\nu_t | \mK^+(\mK^-(\nu_t))) 
    + \frac{3\eta\tau^2}{2} \int_{\mY} \Big|\nabla_y \log\Big(\frac{d \mK^+(\mK^-(\nu_t))}{d \mK^+(\mu_t)}\Big) \Big|^2 d\nu_t(y) \\
    & \leq  \frac{\eta\tau^2}{2} \mI(\nu_t | \mK^+(\mK^-(\nu_t))) 
    + 3\eta K_{xy}^2\diam(\mX)^2 \cdot \mH(\mu_t|\mK^-(\nu_t)). 
\end{aligned}
$$
The estimate \eqref{eq:dmL4} follows directly from the last two inequalities. 
\end{proof}

\begin{proof}[Proof of Part (ii) of Theorem \ref{thm:main1} ] It follows from Proposition \ref{prop:ml24} that 
$$
\begin{aligned}
    \frac{d}{dt} \widetilde{\mL}(\mu_t,\nu_t) & \leq -\frac{\eta\tau^2}{2} (1-\gamma) \mI(\nu_t | \mK^+(\mK^-(\nu_t)))
     - \gamma\tau^2 \mI(\mu_t | \mK^-(\nu_t)) \\
     & + \eta K_{xy}^2\diam(\mX)^2(1+ 3\gamma)\cdot  \mH(\mu_t|\mK^-(\nu_t)). 
\end{aligned}
  $$  
  Thanks to Lemma \ref{lem:LSI} and \eqref{eq:ml3},
$$
\begin{aligned}
    \tau\mI(\nu_t | \mK^+(\mK^-(\nu_t))) & \geq \tau \lambda_{LS} \mH(\nu_t | \mK^+(\mK^-(\nu_t)))\\
    & \geq  \lambda_{LS} \mL_3(\nu_t).
\end{aligned}
$$
As a result of above and \eqref{eq:ml4}, 
$$\begin{aligned}
     \frac{d}{dt} \widetilde{\mL}(\mu_t,\nu_t) & \leq -\frac{\eta\tau}{2} (1-\gamma) \lambda_{LS}  \mL_3(\nu_t)
     - \tau\Big( \gamma \lambda_{LS}  -\eta K_{xy}^2\diam(\mX)^2(1+ 3\gamma)\Big)\mL_4(\mu_t,\nu_t) .
\end{aligned}
$$
Therefore setting 
$$
\eta  = \frac{\gamma\tau^2\lambda_{LS}}{2K_{xy}^2\diam(\mX)^2 (1+ 3\gamma)}
$$
and applying the Gr\"onwall's inequality to above leads to 
$$
\begin{aligned}
    \widetilde{\mL}(\mu_t,\nu_t) \leq e^{-\alpha t} \widetilde{\mL}(\mu_0,\nu_0)
\end{aligned}
$$
with 
$$
\alpha = \frac{\tau\lambda_{LS}}{2} \Big(1\wedge  \eta(1-\gamma)\Big).
$$
\end{proof}

\begin{proof}[Proof of Part (ii) of Theorem \ref{thm:main2}]
    Thanks to Part (ii) of Theorem \ref{thm:main1} and the bound \eqref{eq:ml3}, we have that 
\begin{equation}\label{eq:munut1}
   \tau\mH(\nu_t | \nu^\ast) \leq \widetilde{L}(\mu_t,\nu_t)  \leq e^{-\alpha_1 t} \widetilde{L}(\mu_0,\nu_0)
\end{equation}
and that 
\begin{equation}\label{eq:munut2}
\tau \mH(\mu_t | \mK^- (\nu_t)) = \mL_4 (\mu_t,\nu_t) \leq \frac{e^{-\alpha_1 t}}{\gamma} \widetilde{L}(\mu_0,\nu_0).
\end{equation}
Next to obtain the exponential decay of $\tau \mH(\mu_t | \mu^\ast)$, observe that 
\begin{equation}\begin{aligned}\label{eq:bdnutnu}
        \tau \mH(\mu_t | \mu^\ast) & =  \tau \mH(\mu_t | \mK^- (\nu_t)) + \tau \int_{\mX} (\log \mK^- (\nu_t) - \log \mu^\ast) d\mu_t\\
    & = \tau \mH(\mu_t | \mK^- (\nu_t)) + \tau \int_{\mX} (\log \mK^- (\nu_t) - \log \mu^\ast) d(\mu_t - \mu^\ast) - \tau \mH(\mu^\ast | \mK^- (\nu_t))\\
    & \leq \tau \mH(\mu_t | \mK^- (\nu_t)) + \tau \int_{\mX} (\log \mK^- (\nu_t) - \log \mu^\ast) d(\mu_t - \mu^\ast).
\end{aligned}
\end{equation}
Since $\mu^\ast = \mK^- (\nu^\ast)$, we can write 
$$\begin{aligned}
    \log \mK^- (\nu_t) - \log \mu^\ast  = -\tau^{-1} \int_{\mY} K(x,y) (\nu_t - \nu^\ast)(dy) - (\log Z^-(\nu_t) - \log Z^-(\nu^\ast)).
\end{aligned}
$$
Hence the second term on the RHS of the last line of \eqref{eq:bdnutnu} can be bounded by in a similar manner as \eqref{eq:bdmutmu}. Namely, 
$$
\begin{aligned}
    \tau \int_{\mX} (\log \mK^- (\nu_t) - \log \mu^\ast) d(\mu_t - \mu^\ast)& = \int_{\mX}\int_{\mY} K(x,y) (\nu_t - \nu^\ast)(dy)(\mu_t - \mu^\ast)(dx)\\
    & \leq \frac{\tau}{2} \mH(\mu_t | \mu^\ast) + \frac{2 \|K\|_{\infty}^2}{\tau} \mH(\nu_t | \nu^\ast).
\end{aligned}
$$
Combining the last inequality with \eqref{eq:bdnutnu} leads to 
$$\begin{aligned}
    \tau \mH(\mu_t | \mu^\ast) & \leq 2\tau \mH(\mu_t | \mK^- (\nu_t))  + \frac{4 \|K\|_{\infty}^2}{\tau} \mH(\nu_t | \nu^\ast)\\
    & \leq \Big(\frac{2}{\gamma} + \frac{4 \|K\|_{\infty}^2}{\tau^2}\Big)e^{-\alpha_2 t}\widetilde{L}(\mu_0,\nu_0),
\end{aligned}
$$
where we have used \eqref{eq:munut1} and \eqref{eq:munut2} in the last inequality. 
\end{proof}

\section{Proof of convergence for the annealed dynamics in the fast descent regime}
\label{sec:appsa}

\begin{proof}[Proof of Part (ii) of Theorem \ref{thm:sa1}]
By a direct calculation, one has 
\begin{equation}
    \begin{aligned}
        \frac{d}{dt} \mL_3(\nu_t) \leq \frac{\tau_t^\prime}{\tau_t} (\mL_3(\nu_t) + 2\|K\|_{\infty}) -\frac{\eta_t \tau_t^2}{2} \mI(\nu_t | \mK^+(\mK^-(\nu_t))) + \eta_t K_{xy}^2 \diam(\mX)^2 \mH(\mu_t | \mK^-(\nu_t))
    \end{aligned}
\end{equation}
and 
\begin{equation}
    \begin{aligned}
        \frac{d}{dt} \mL_4(\mu_t, \nu_t)& \leq 
        \frac{\tau_t^\prime}{\tau_t} (\mL_4(\nu_t) + 2\|K\|_{\infty})
        -\tau_t^2 \mI(\mu_t | \mK^-(\nu_t)) + \frac{\tau_t^2\eta_t}{2} \mI(\nu_t | \mK^+(\mK^-(\nu_t)))\\
        & +  3\eta_t K_{xy}^2 \diam(\mX)^2 \mH(\mu_t | \mK^-(\nu_t)).
    \end{aligned}
\end{equation}
Combining the last two displays, one can obtain the time-derivative of $\widetilde{\mL}(\mu_t,\nu_t)$ as follows
$$
\begin{aligned}
    \frac{d}{dt} \widetilde{\mL}(\mu_t,\nu_t) & \leq \frac{\tau_t^\prime}{\tau_t} \widetilde{\mL}(\mu_t,\nu_t) + \frac{\tau_t^\prime}{\tau_t} 2(1+\gamma) \|K\|_{\infty} -\gamma\tau_t^2 \mI(\mu_t | \mK^-(\nu_t)) \\
    & -\frac{(1-\gamma)\eta_t \tau_t^2}{2} \mI(\nu_t | \mK^+(\mK^-(\nu_t)))
    +\eta_t(1+3\gamma) K_{xy}^2 \diam(\mX)^2 \mH(\mu_t | \mK^-(\nu_t))\\
    & \leq \frac{\tau_t^\prime}{\tau_t} \widetilde{\mL}(\mu_t,\nu_t) - \frac{(1-\gamma)\eta_t\tau_t \lambda_{LS}(\tau_t)}{2} \mL_3(\nu_t) - \tau_t\Big(\gamma\lambda_{LS}(\tau_t)\\
    & \qquad -  \frac{\eta_t(1+3\gamma) K_{xy}^2 \diam(\mX)^2}{\tau_t^2}\Big)\mL_4(\mu_t,\nu_t) + \frac{2\tau_t^\prime}{\tau_t}(1+\gamma) \|K\|_{\infty}\\
    & \leq -\alpha_2(t) \widetilde{\mL}(\mu_t,\nu_t)+ \frac{2\tau_t^\prime}{\tau_t} (1+\gamma) \|K\|_{\infty},
\end{aligned}
$$
where 
$$\begin{aligned}
\alpha_2(t) & = \frac{(1-\gamma)\eta_t\tau_t \lambda_{LS}(\tau_t)}{2} \wedge \tau_t \Big(\lambda_{LS}(\tau_t) - \frac{\eta_t(1+3\gamma) K_{xy}^2 \diam(\mX)^2}{\gamma\tau_t^2}-\frac{\tau_t^\prime}{\tau_t}\Big) \\
& \geq \frac{\xi}{\log t} t^{-(\xi^\ast/\xi)}\Big( C_1\eta_t \wedge \Big(t^{-(\xi^\ast/\xi)}   + \frac{C_2}{t\log t}-\frac{C_3 \eta_t}{\log(t)^2}\Big)\Big).
\end{aligned}$$
Setting $\eta_t = c\log t/t$ for some $c< C_2/C_3$ in the above yields that for every $0<\epsilon<1-\xi^\ast/\xi$ and $t$ large enough 
 $$
 \alpha_2(t) \geq Ct^{-2\xi^\ast/\xi-\epsilon}.
 $$
Therefore we have obtained that 
$$
\begin{aligned}
    \frac{d}{dt} \widetilde{\mL}(\mu_t,\nu_t) \leq 
    -Ct^{-2\xi^\ast/\xi-\epsilon} \widetilde{\mL}(\mu_t,\nu_t) + \frac{C^\prime}{t\log t}. 
\end{aligned}
$$
Similar to the proof of \eqref{eq:bdLmunu}, one can obtain from above that for $t$ large enough
$$
\widetilde{\mL}(\mu_t,\nu_t) \leq  C^{\prime\prime}t^{-1 + 2\xi^\ast/\xi+\epsilon}.
$$
This directly implies the entropy decay bounds in \eqref{eq:Lmutstar2}. Finally the  estimate  \eqref{eq:NImut2} follows from \eqref{eq:Lmutstar2} and the arguments used in  Step 2 of the proof of Theorem \ref{thm:sa1}- Part(i).
\end{proof}
\bibliographystyle{plain}
\bibliography{references}

\begin{thebibliography}{10}

\bibitem{abernethy2019last}
Jacob Abernethy, Kevin~A Lai, and Andre Wibisono.
\newblock Last-iterate convergence rates for min-max optimization.
\newblock {\em arXiv preprint arXiv:1906.02027}, 2019.

\bibitem{araujo2019mean}
Dyego Ara{\'u}jo, Roberto~I Oliveira, and Daniel Yukimura.
\newblock A mean-field limit for certain deep neural networks.
\newblock {\em arXiv preprint arXiv:1906.00193}, 2019.

\bibitem{bach2017breaking}
Francis Bach.
\newblock Breaking the curse of dimensionality with convex neural networks.
\newblock {\em The Journal of Machine Learning Research}, 18(1):629--681, 2017.

\bibitem{barron1993universal}
Andrew~R Barron.
\newblock Universal approximation bounds for superpositions of a sigmoidal
  function.
\newblock {\em IEEE Transactions on Information theory}, 39(3):930--945, 1993.

\bibitem{brenier2020optimal}
Yann Brenier and Dmitry Vorotnikov.
\newblock On optimal transport of matrix-valued measures.
\newblock {\em SIAM Journal on Mathematical Analysis}, 52(3):2849--2873, 2020.

\bibitem{busoniu2008comprehensive}
Lucian Busoniu, Robert Babuska, and Bart De~Schutter.
\newblock A comprehensive survey of multiagent reinforcement learning.
\newblock {\em IEEE Transactions on Systems, Man, and Cybernetics, Part C
  (Applications and Reviews)}, 38(2):156--172, 2008.

\bibitem{chizat2022mean}
L{\'e}na{\"\i}c Chizat.
\newblock Mean-field langevin dynamics: Exponential convergence and annealing.
\newblock {\em arXiv preprint arXiv:2202.01009}, 2022.

\bibitem{chizat2018global}
Lenaic Chizat and Francis Bach.
\newblock On the global convergence of gradient descent for over-parameterized
  models using optimal transport.
\newblock {\em Advances in neural information processing systems}, 31, 2018.

\bibitem{chizat2018interpolating}
Lenaic Chizat, Gabriel Peyr{\'e}, Bernhard Schmitzer, and Fran{\c{c}}ois-Xavier
  Vialard.
\newblock An interpolating distance between optimal transport and fisher--rao
  metrics.
\newblock {\em Foundations of Computational Mathematics}, 18(1):1--44, 2018.

\bibitem{dai2019exponential}
Bo~Dai, Zhen Liu, Hanjun Dai, Niao He, Arthur Gretton, Le~Song, and Dale
  Schuurmans.
\newblock Exponential family estimation via adversarial dynamics embedding.
\newblock {\em Advances in Neural Information Processing Systems}, 32, 2019.

\bibitem{daskalakis2018limit}
Constantinos Daskalakis and Ioannis Panageas.
\newblock The limit points of (optimistic) gradient descent in min-max
  optimization.
\newblock {\em Advances in neural information processing systems}, 31, 2018.

\bibitem{daskalakis2019last}
Constantinos Daskalakis and Ioannis Panageas.
\newblock Last-iterate convergence: Zero-sum games and constrained min-max
  optimization.
\newblock {\em 10th Innovations in Theoretical Computer Science}, 2019.

\bibitem{doan2022convergence}
Thinh Doan.
\newblock Convergence rates of two-time-scale gradient descent-ascent dynamics
  for solving nonconvex min-max problems.
\newblock In {\em Learning for Dynamics and Control Conference}, pages
  192--206. PMLR, 2022.

\bibitem{domingo2021dual}
Carles Domingo-Enrich, Alberto Bietti, Marylou Gabri{\'e}, Joan Bruna, and Eric
  Vanden-Eijnden.
\newblock Dual training of energy-based models with overparametrized shallow
  neural networks.
\newblock {\em arXiv preprint arXiv:2107.05134}, 2021.

\bibitem{domingo2022simultaneous}
Carles Domingo-Enrich and Joan Bruna.
\newblock Simultaneous transport evolution for minimax equilibria on measures.
\newblock {\em arXiv preprint arXiv:2202.06460}, 2022.

\bibitem{domingo2020mean}
Carles Domingo-Enrich, Samy Jelassi, Arthur Mensch, Grant Rotskoff, and Joan
  Bruna.
\newblock A mean-field analysis of two-player zero-sum games.
\newblock {\em Advances in neural information processing systems},
  33:20215--20226, 2020.

\bibitem{eberle2019quantitative}
Andreas Eberle, Arnaud Guillin, and Raphael Zimmer.
\newblock Quantitative harris-type theorems for diffusions and mckean--vlasov
  processes.
\newblock {\em Transactions of the American Mathematical Society},
  371(10):7135--7173, 2019.

\bibitem{funaki1984certain}
Tadahisa Funaki.
\newblock A certain class of diffusion processes associated with nonlinear
  parabolic equations.
\newblock {\em Zeitschrift f\" ur Wahrscheinlichkeitstheorie und Verwandte
  Gebiete}, 67(3):331--348, 1984.

\bibitem{glicksberg1952further}
Irving~L Glicksberg.
\newblock A further generalization of the kakutani fixed point theorem, with
  application to nash equilibrium points.
\newblock {\em Proceedings of the American Mathematical Society},
  3(1):170--174, 1952.

\bibitem{goodfellow2020generative}
Ian Goodfellow, Jean Pouget-Abadie, Mehdi Mirza, Bing Xu, David Warde-Farley,
  Sherjil Ozair, Aaron Courville, and Yoshua Bengio.
\newblock Generative adversarial networks.
\newblock {\em Communications of the ACM}, 63(11):139--144, 2020.

\bibitem{heusel2017gans}
Martin Heusel, Hubert Ramsauer, Thomas Unterthiner, Bernhard Nessler, and Sepp
  Hochreiter.
\newblock Gans trained by a two time-scale update rule converge to a local nash
  equilibrium.
\newblock {\em Advances in neural information processing systems}, 30, 2017.

\bibitem{hsieh2019finding}
Ya-Ping Hsieh, Chen Liu, and Volkan Cevher.
\newblock Finding mixed nash equilibria of generative adversarial networks.
\newblock In {\em International Conference on Machine Learning}, pages
  2810--2819. PMLR, 2019.

\bibitem{hu2021mean}
Kaitong Hu, Zhenjie Ren, David {\v{S}}i{\v{s}}ka, and {\L}ukasz Szpruch.
\newblock Mean-field langevin dynamics and energy landscape of neural networks.
\newblock In {\em Annales de l'Institut Henri Poincar{\'e}, Probabilit{\'e}s et
  Statistiques}, volume~57, pages 2043--2065. Institut Henri Poincar{\'e},
  2021.

\bibitem{jin2020local}
Chi Jin, Praneeth Netrapalli, and Michael Jordan.
\newblock What is local optimality in nonconvex-nonconcave minimax
  optimization?
\newblock In {\em International conference on machine learning}, pages
  4880--4889. PMLR, 2020.

\bibitem{kondratyev2016new}
Stanislav Kondratyev, L{\'e}onard Monsaingeon, and Dmitry Vorotnikov.
\newblock A new optimal transport distance on the space of finite radon
  measures.
\newblock {\em Advances in Differential Equations}, 21(11/12):1117--1164, 2016.

\bibitem{kondratyev2019spherical}
Stanislav Kondratyev and Dmitry Vorotnikov.
\newblock Spherical hellinger--kantorovich gradient flows.
\newblock {\em SIAM Journal on Mathematical Analysis}, 51(3):2053--2084, 2019.

\bibitem{laschos2019geometric}
Vaios Laschos and Alexander Mielke.
\newblock Geometric properties of cones with applications on the
  hellinger--kantorovich space, and a new distance on the space of probability
  measures.
\newblock {\em Journal of Functional Analysis}, 276(11):3529--3576, 2019.

\bibitem{lin2020gradient}
Tianyi Lin, Chi Jin, and Michael Jordan.
\newblock On gradient descent ascent for nonconvex-concave minimax problems.
\newblock In {\em International Conference on Machine Learning}, pages
  6083--6093. PMLR, 2020.

\bibitem{lu2020mean}
Yiping Lu, Chao Ma, Yulong Lu, Jianfeng Lu, and Lexing Ying.
\newblock A mean field analysis of deep resnet and beyond: Towards provably
  optimization via overparameterization from depth.
\newblock In {\em International Conference on Machine Learning}, pages
  6426--6436. PMLR, 2020.

\bibitem{lu2019accelerating}
Yulong Lu, Jianfeng Lu, and James Nolen.
\newblock Accelerating langevin sampling with birth-death.
\newblock {\em arXiv preprint arXiv:1905.09863}, 2019.

\bibitem{lu2022birth}
Yulong Lu, Dejan Slep{\v{c}}ev, and Lihan Wang.
\newblock Birth-death dynamics for sampling: Global convergence, approximations
  and their asymptotics.
\newblock {\em arXiv preprint arXiv:2211.00450}, 2022.

\bibitem{ma2022barron}
Chao Ma, Lei Wu, et~al.
\newblock The barron space and the flow-induced function spaces for neural
  network models.
\newblock {\em Constructive Approximation}, 55(1):369--406, 2022.

\bibitem{ma2021provably}
Chao Ma and Lexing Ying.
\newblock Provably convergent quasistatic dynamics for mean-field two-player
  zero-sum games.
\newblock In {\em International Conference on Learning Representations}, 2021.

\bibitem{madry2018towards}
Aleksander Madry, Aleksandar Makelov, Ludwig Schmidt, Dimitris Tsipras, and
  Adrian Vladu.
\newblock Towards deep learning models resistant to adversarial attacks.
\newblock In {\em International Conference on Learning Representations}, 2018.

\bibitem{mei2018mean}
Song Mei, Andrea Montanari, and Phan-Minh Nguyen.
\newblock A mean field view of the landscape of two-layer neural networks.
\newblock {\em Proceedings of the National Academy of Sciences},
  115(33):E7665--E7671, 2018.

\bibitem{mertikopoulos2018optimistic}
Panayotis Mertikopoulos, Bruno Lecouat, Houssam Zenati, Chuan-Sheng Foo, Vijay
  Chandrasekhar, and Georgios Piliouras.
\newblock Optimistic mirror descent in saddle-point problems: Going the extra
  (gradient) mile.
\newblock {\em arXiv preprint arXiv:1807.02629}, 2018.

\bibitem{mitchell1980finite}
Andrew~Ronald Mitchell and David~Francis Griffiths.
\newblock The finite difference method in partial differential equations.
\newblock {\em A Wiley-Interscience Publication}, 1980.

\bibitem{Nash1951}
John Nash.
\newblock Non-cooperative games.
\newblock {\em Annals of Mathematics}, 54:286--295, 1951.

\bibitem{nguyen2019mean}
Phan-Minh Nguyen.
\newblock Mean field limit of the learning dynamics of multilayer neural
  networks.
\newblock {\em arXiv preprint arXiv:1902.02880}, 2019.

\bibitem{nikaido1954neumann}
Hukukane Nikaido.
\newblock On von neumann's minimax theorem.
\newblock {\em Pacific J. Math}, 4(1):65--72, 1954.

\bibitem{nikaido1955note}
Hukukane Nikaido and Kazuo Isoda.
\newblock Note on non-cooperative convex games.
\newblock {\em Pacific Journal of Mathematics}, 5(S1):807--815, 1955.

\bibitem{nitanda2022convex}
Atsushi Nitanda, Denny Wu, and Taiji Suzuki.
\newblock Convex analysis of the mean field langevin dynamics.
\newblock In {\em International Conference on Artificial Intelligence and
  Statistics}, pages 9741--9757. PMLR, 2022.

\bibitem{Petrov40}
Georgii~Ivanovich Petrov.
\newblock Application of the method of galerkin to a problem involving the
  stationary flow of a viscous fluid.
\newblock {\em Prikl. Matem. Mekh.}, 4(3):3--12, 1940.

\bibitem{raginsky2017non}
Maxim Raginsky, Alexander Rakhlin, and Matus Telgarsky.
\newblock Non-convex learning via stochastic gradient langevin dynamics: a
  nonasymptotic analysis.
\newblock In {\em Conference on Learning Theory}, pages 1674--1703. PMLR, 2017.

\bibitem{rotskoff2019global}
Grant Rotskoff, S~Jelassi, Joan Bruna, and Eric Vanden-Eijnden.
\newblock Global convergence of neuron birth-death dynamics.
\newblock 2019.

\bibitem{rotskoff2022trainability}
Grant Rotskoff and Eric Vanden-Eijnden.
\newblock Trainability and accuracy of artificial neural networks: An
  interacting particle system approach.
\newblock {\em Communications on Pure and Applied Mathematics},
  75(9):1889--1935, 2022.

\bibitem{sion1958general}
Maurice Sion.
\newblock On general minimax theorems.
\newblock {\em Pacific Journal of mathematics}, 8(1):171--176, 1958.

\bibitem{sirignano2020mean}
Justin Sirignano and Konstantinos Spiliopoulos.
\newblock Mean field analysis of neural networks: A law of large numbers.
\newblock {\em SIAM Journal on Applied Mathematics}, 80(2):725--752, 2020.

\bibitem{sirignano2022mean}
Justin Sirignano and Konstantinos Spiliopoulos.
\newblock Mean field analysis of deep neural networks.
\newblock {\em Mathematics of Operations Research}, 47(1):120--152, 2022.

\bibitem{sznitman1991topics}
Alain-Sol Sznitman.
\newblock Topics in propagation of chaos.
\newblock In {\em Ecole d'Et\' e $de$ $probabilit\'es$ $de$ Saint-Flour
  XIX-1989}, pages 165--251. Springer, 1991.

\bibitem{tang2021simulated}
Wenpin Tang and Xun~Yu Zhou.
\newblock Simulated annealing from continuum to discretization: a convergence
  analysis via the eyring--kramers law.
\newblock {\em arXiv preprint arXiv:2102.02339}, 2021.

\bibitem{v1928theorie}
J~Von.~Neumann.
\newblock Zur theorie der gesellschaftsspiele.
\newblock {\em Mathematische annalen}, 100(1):295--320, 1928.

\bibitem{wang2022exponentially}
Guillaume Wang and L{\'e}na{\"\i}c Chizat.
\newblock An exponentially converging particle method for the mixed nash
  equilibrium of continuous games.
\newblock {\em arXiv preprint arXiv:2211.01280}, 2022.

\bibitem{wojtowytsch2020banach}
Stephan Wojtowytsch et~al.
\newblock On the banach spaces associated with multi-layer relu networks:
  Function representation, approximation theory and gradient descent dynamics.
\newblock {\em arXiv preprint arXiv:2007.15623}, 2020.

\bibitem{yang2020global}
Junchi Yang, Negar Kiyavash, and Niao He.
\newblock Global convergence and variance reduction for a class of
  nonconvex-nonconcave minimax problems.
\newblock {\em Advances in Neural Information Processing Systems},
  33:1153--1165, 2020.

\bibitem{yang2022faster}
Junchi Yang, Antonio Orvieto, Aurelien Lucchi, and Niao He.
\newblock Faster single-loop algorithms for minimax optimization without strong
  concavity.
\newblock In {\em International Conference on Artificial Intelligence and
  Statistics}, pages 5485--5517. PMLR, 2022.

\bibitem{zang2020weak}
Yaohua Zang, Gang Bao, Xiaojing Ye, and Haomin Zhou.
\newblock Weak adversarial networks for high-dimensional partial differential
  equations.
\newblock {\em Journal of Computational Physics}, 411:109409, 2020.

\end{thebibliography}
\end{document}